\documentclass[a4paper, 11pt]{amsart}
\usepackage[a4paper,margin=3cm]{geometry}

\usepackage{amssymb,amsfonts,amsmath,amsthm,yhmath, stmaryrd}
\usepackage{tikz}
\usepackage{wrapfig}
\usetikzlibrary{arrows.meta,decorations.markings}
\definecolor{FigureBlue}{rgb}{0.165,0.510,0.980}
\definecolor{LinkBlue}{rgb}{0.0,0.0,0.5}
\newtheorem{thm}{Theorem}
\newtheorem{lemma}[thm]{Lemma}
\newtheorem{cor}[thm]{Corollary}
\theoremstyle{definition}
\newtheorem{prop}[thm]{Proposition}
\newtheorem{remark}[thm]{Remark}
\theoremstyle{plain}
\def\RR{\mathbb{R}}
\def\ZZ{\mathbb{Z}}

\def\QQ{\mathbb{Q}}

\def\NN{\mathbb{N}}
\def\muturn{\mu_{\hbox{turn}}}
\newcommand{\Kappa}{\mathcal{K}}

\def\Mcal{\mathcal{M}}
\def\Ocal{\mathcal{O}}

\newcommand{\MRnorm}[1]{\left\lVert #1\right\rVert_{\mathrm{X}}}
\newcommand{\infnorm}[1]{\left\lVert #1\right\rVert_{\infty}}

\newcommand{\abs}[1]{\left\lvert #1\right\rvert}
\newcommand{\vol}[1]{\operatorname{vol}#1}
\newcommand{\diam}[1]{\operatorname{diam}#1}

\DeclareMathOperator{\trace}{Tr}
\DeclareMathOperator{\norm}{N}

\def\MR#1{\href{https://mathscinet.ams.org/mathscinet/relay-station?mr=#1}{MR#1}} 
\usepackage{mathtools}
\mathtoolsset{showonlyrefs}

\usepackage[
   pdfstartview=FitH,
   pdfpagemode=UseNone,
   colorlinks=true,
   allcolors=LinkBlue,
]{hyperref}
\newcommand{\MRBallBoundary}{%
  (0.519521730308756884400330651529448941935548067,0.000000000000000000000000000000000000000000000) coordinate (MRball-1-0) 
  -- (0.526846097343530189103104624263377640778963633,0.037631864095966442078793187447384117198497402) coordinate (MRball-14-1) 
  -- (0.527418073647430503643650573808505177148067401,0.040570621049802346434126967216038859780620569) coordinate (MRball-13-1) 
  -- (0.528086951002595462711202682071936815647387846,0.044007245916882955225933556839328067970615654) coordinate (MRball-12-1) 
  -- (0.528879632972240553162026100304349907958246270,0.048079966633840050287456918209486355268931479) coordinate (MRball-11-1) 
  -- (0.529833998002099238611748077020204915032999759,0.052983399800209923861174807702020491503299976) coordinate (MRball-10-1) 
  -- (0.531005130353103119834749646957188693521722138,0.059000570039233679981638849661909854835746904) coordinate (MRball-9-1) 
  -- (0.532476334169215342951909054565463300286074081,0.066559541771151917868988631820682912535759260) coordinate (MRball-8-1) 
  -- (0.534379806106011442682265518122456191584111434,0.076339972300858777526037931160350884512015919) coordinate (MRball-7-1) 
  -- (0.535557591296605045618463937717183099268718551,0.082393475584093083941302144264182015272110546) coordinate (MRball-13-2) 
  -- (0.536938252694525464289533392043224242467016921,0.089489708782420910714922232007204040411169487) coordinate (MRball-6-1) 
  -- (0.538576189996503483005045268935364267783495359,0.097922943635727906000917321624611685051544611) coordinate (MRball-11-2) 
  -- (0.540554953483433104885693085991909268071174038,0.108110990696686620977138617198381853614234808) coordinate (MRball-5-1) 
  -- (0.542103781718258844061928158134859442960445055,0.116165096082484038013270319600327023491523940) coordinate (MRball-14-3) 
  -- (0.542968083088059569885806403773548782922430901,0.120659574019568793307956978616344173982762422) coordinate (MRball-9-2) 
  -- (0.543901945103062936287348547185389844270128195,0.125515833485322216066311203196628425600798814) coordinate (MRball-13-3) 
  -- (0.546014925347583108925018390056453813515839022,0.136503731336895777231254597514113453378959755) coordinate (MRball-4-1) 
  -- (0.548389324493065640142701258966104301767622331,0.149560724861745174584373070627119355027533363) coordinate (MRball-11-3) 
  -- (0.549755418133313758192063726361281838059235646,0.157072976609518216626303921817509096588353042) coordinate (MRball-7-2) 
  -- (0.551265286106721020743332982099411050784101434,0.165379585832016306222999894629823315235230430) coordinate (MRball-10-3) 
  -- (0.552081729725039789694966490344352569109809033,0.169871301453858396829220458567493098187633549) coordinate (MRball-13-4) 
  -- (0.554820761022722509499003315631330519773714161,0.184940253674240836499667771877110173257904720) coordinate (MRball-3-1) 
  -- (0.556593599297261798099717039871956391354093908,0.198783428320450642178470371382841568340747824) coordinate (MRball-14-5) 
  -- (0.557079068655728892798269367283279542830616151,0.202574206783901415563007042648465288302042237) coordinate (MRball-11-4) 
  -- (0.557930680095718723515006804480638716730810108,0.209224005035894521318127551680239518774053791) coordinate (MRball-8-3) 
  -- (0.558653196889898486901395168126903837358327650,0.214866614188422494962075064664193783599356788) coordinate (MRball-13-5) 
  -- (0.559813124126872786755840902881935562491859993,0.223925249650749114702336361152774224996743997) coordinate (MRball-5-2) 
  -- (0.561047426167051703844817489790681434139747313,0.233769760902938209935340620746117264224894714) coordinate (MRball-12-5) 
  -- (0.561932407935921661617489361889114596667503816,0.240828174829680712121781155095334827143215921) coordinate (MRball-7-3) 
  -- (0.563115639172105336662626541080197361549888915,0.250273617409824594072278462702309938466617296) coordinate (MRball-9-4) 
  -- (0.563871175924097787224163329067816687074810262,0.256305079965498994192801513212643948670368301) coordinate (MRball-11-5) 
  -- (0.564395427932577660605210232978648363205293637,0.260490197507343535663943184451683859940904755) coordinate (MRball-13-6) 
  -- (0.567296328553255492028622702576643118931495396,0.283648164276627746014311351288321559465747698) coordinate (MRball-2-1) 
  -- (0.564395427932577660605210232978648363205293637,0.303905230425234124941267048526964503264388881) coordinate (MRball-13-7) 
  -- (0.563871175924097787224163329067816687074810262,0.307566095958598793031361815855172738404441961) coordinate (MRball-11-6) 
  -- (0.563115639172105336662626541080197361549888915,0.312842021762280742590348078377887423083271620) coordinate (MRball-9-5) 
  -- (0.561932407935921661617489361889114596667503816,0.321104233106240949495708206793779769524287895) coordinate (MRball-7-4) 
  -- (0.561047426167051703844817489790681434139747313,0.327277665264113493909476869044564169914852599) coordinate (MRball-12-7) 
  -- (0.559813124126872786755840902881935562491859993,0.335887874476123672053504541729161337495115996) coordinate (MRball-5-3) 
  -- (0.558653196889898486901395168126903837358327650,0.343786582701475991939320103462710053758970861) coordinate (MRball-13-8) 
  -- (0.557930680095718723515006804480638716730810108,0.348706675059824202196879252800399197956756318) coordinate (MRball-8-5) 
  -- (0.557079068655728892798269367283279542830616151,0.354504861871827477235262324634814254528573914) coordinate (MRball-11-7) 
  -- (0.556593599297261798099717039871956391354093908,0.357810170976811155921246668489114823013346083) coordinate (MRball-14-9) 
  -- (0.554820761022722509499003315631330519773714161,0.369880507348481672999335543754220346515809441) coordinate (MRball-3-2) 
  -- (0.552081729725039789694966490344352569109809033,0.382210428271181392865746031776859470922175484) coordinate (MRball-13-9) 
  -- (0.551265286106721020743332982099411050784101434,0.385885700274704714520333087469587735548871004) coordinate (MRball-10-7) 
  -- (0.549755418133313758192063726361281838059235646,0.392682441523795541565759804543772741470882604) coordinate (MRball-7-5) 
  -- (0.548389324493065640142701258966104301767622331,0.398828599631320465558328188338984946740088968) coordinate (MRball-11-8) 
  -- (0.546014925347583108925018390056453813515839022,0.409511194010687331693763792542340360136879266) coordinate (MRball-4-3) 
  -- (0.543901945103062936287348547185389844270128195,0.418386111617740720221037343988761418669329380) coordinate (MRball-13-10) 
  -- (0.542968083088059569885806403773548782922430901,0.422308509068490776577849425157204608939668478) coordinate (MRball-9-7) 
  -- (0.542103781718258844061928158134859442960445055,0.425938685635774806048657838534532419468921115) coordinate (MRball-14-11) 
  -- (0.540554953483433104885693085991909268071174038,0.432443962786746483908554468793527414456939230) coordinate (MRball-5-4) 
  -- (0.538576189996503483005045268935364267783495359,0.440653246360775577004127947310752582731950748) coordinate (MRball-11-9) 
  -- (0.536938252694525464289533392043224242467016921,0.447448543912104553574611160036020202055847434) coordinate (MRball-6-5) 
  -- (0.535557591296605045618463937717183099268718551,0.453164115712511961677161793453001083996608005) coordinate (MRball-13-11) 
  -- (0.534379806106011442682265518122456191584111434,0.458039833805152665156227586962105307072095515) coordinate (MRball-7-6) 
  -- (0.532476334169215342951909054565463300286074081,0.465916792398063425082920422744780387750314821) coordinate (MRball-8-7) 
  -- (0.531005130353103119834749646957188693521722138,0.472004560313869439853110797295278838685975234) coordinate (MRball-9-8) 
  -- (0.529833998002099238611748077020204915032999759,0.476850598201889314750573269318184423529699783) coordinate (MRball-10-9) 
  -- (0.528879632972240553162026100304349907958246270,0.480799666338400502874569182094863552689314791) coordinate (MRball-11-10) 
  -- (0.528086951002595462711202682071936815647387846,0.484079705085712507485269125232608747676772192) coordinate (MRball-12-11) 
  -- (0.527418073647430503643650573808505177148067401,0.486847452597628157209523606592466317367446831) coordinate (MRball-13-12) 
  -- (0.526846097343530189103104624263377640778963633,0.489214233247563747024311436815993523580466231) coordinate (MRball-14-13) 
  -- (0.519521730308756884400330651529448941935548067,0.519521730308756884400330651529448941935548067) coordinate (MRball-1-1) 
  -- (0.489214233247563747024311436815993523580466231,0.526846097343530189103104624263377640778963633) coordinate (MRball-13-14) 
  -- (0.486847452597628157209523606592466317367446831,0.527418073647430503643650573808505177148067401) coordinate (MRball-12-13) 
  -- (0.484079705085712507485269125232608747676772192,0.528086951002595462711202682071936815647387846) coordinate (MRball-11-12) 
  -- (0.480799666338400502874569182094863552689314791,0.528879632972240553162026100304349907958246270) coordinate (MRball-10-11) 
  -- (0.476850598201889314750573269318184423529699783,0.529833998002099238611748077020204915032999759) coordinate (MRball-9-10) 
  -- (0.472004560313869439853110797295278838685975234,0.531005130353103119834749646957188693521722138) coordinate (MRball-8-9) 
  -- (0.465916792398063425082920422744780387750314821,0.532476334169215342951909054565463300286074081) coordinate (MRball-7-8) 
  -- (0.458039833805152665156227586962105307072095515,0.534379806106011442682265518122456191584111434) coordinate (MRball-6-7) 
  -- (0.453164115712511961677161793453001083996608005,0.535557591296605045618463937717183099268718551) coordinate (MRball-11-13) 
  -- (0.447448543912104553574611160036020202055847434,0.536938252694525464289533392043224242467016921) coordinate (MRball-5-6) 
  -- (0.440653246360775577004127947310752582731950748,0.538576189996503483005045268935364267783495359) coordinate (MRball-9-11) 
  -- (0.432443962786746483908554468793527414456939230,0.540554953483433104885693085991909268071174038) coordinate (MRball-4-5) 
  -- (0.425938685635774806048657838534532419468921115,0.542103781718258844061928158134859442960445055) coordinate (MRball-11-14) 
  -- (0.422308509068490776577849425157204608939668478,0.542968083088059569885806403773548782922430901) coordinate (MRball-7-9) 
  -- (0.418386111617740720221037343988761418669329380,0.543901945103062936287348547185389844270128195) coordinate (MRball-10-13) 
  -- (0.409511194010687331693763792542340360136879266,0.546014925347583108925018390056453813515839022) coordinate (MRball-3-4) 
  -- (0.398828599631320465558328188338984946740088968,0.548389324493065640142701258966104301767622331) coordinate (MRball-8-11) 
  -- (0.392682441523795541565759804543772741470882604,0.549755418133313758192063726361281838059235646) coordinate (MRball-5-7) 
  -- (0.385885700274704714520333087469587735548871004,0.551265286106721020743332982099411050784101434) coordinate (MRball-7-10) 
  -- (0.382210428271181392865746031776859470922175484,0.552081729725039789694966490344352569109809033) coordinate (MRball-9-13) 
  -- (0.369880507348481672999335543754220346515809441,0.554820761022722509499003315631330519773714161) coordinate (MRball-2-3) 
  -- (0.357810170976811155921246668489114823013346083,0.556593599297261798099717039871956391354093908) coordinate (MRball-9-14) 
  -- (0.354504861871827477235262324634814254528573914,0.557079068655728892798269367283279542830616151) coordinate (MRball-7-11) 
  -- (0.348706675059824202196879252800399197956756318,0.557930680095718723515006804480638716730810108) coordinate (MRball-5-8) 
  -- (0.343786582701475991939320103462710053758970861,0.558653196889898486901395168126903837358327650) coordinate (MRball-8-13) 
  -- (0.335887874476123672053504541729161337495115996,0.559813124126872786755840902881935562491859993) coordinate (MRball-3-5) 
  -- (0.327277665264113493909476869044564169914852599,0.561047426167051703844817489790681434139747313) coordinate (MRball-7-12) 
  -- (0.321104233106240949495708206793779769524287895,0.561932407935921661617489361889114596667503816) coordinate (MRball-4-7) 
  -- (0.312842021762280742590348078377887423083271620,0.563115639172105336662626541080197361549888915) coordinate (MRball-5-9) 
  -- (0.307566095958598793031361815855172738404441961,0.563871175924097787224163329067816687074810262) coordinate (MRball-6-11) 
  -- (0.303905230425234124941267048526964503264388881,0.564395427932577660605210232978648363205293637) coordinate (MRball-7-13) 
  -- (0.283648164276627746014311351288321559465747698,0.567296328553255492028622702576643118931495396) coordinate (MRball-1-2) 
  -- (0.260490197507343535663943184451683859940904755,0.564395427932577660605210232978648363205293637) coordinate (MRball-6-13) 
  -- (0.256305079965498994192801513212643948670368301,0.563871175924097787224163329067816687074810262) coordinate (MRball-5-11) 
  -- (0.250273617409824594072278462702309938466617296,0.563115639172105336662626541080197361549888915) coordinate (MRball-4-9) 
  -- (0.240828174829680712121781155095334827143215921,0.561932407935921661617489361889114596667503816) coordinate (MRball-3-7) 
  -- (0.233769760902938209935340620746117264224894714,0.561047426167051703844817489790681434139747313) coordinate (MRball-5-12) 
  -- (0.223925249650749114702336361152774224996743997,0.559813124126872786755840902881935562491859993) coordinate (MRball-2-5) 
  -- (0.214866614188422494962075064664193783599356788,0.558653196889898486901395168126903837358327650) coordinate (MRball-5-13) 
  -- (0.209224005035894521318127551680239518774053791,0.557930680095718723515006804480638716730810108) coordinate (MRball-3-8) 
  -- (0.202574206783901415563007042648465288302042237,0.557079068655728892798269367283279542830616151) coordinate (MRball-4-11) 
  -- (0.198783428320450642178470371382841568340747824,0.556593599297261798099717039871956391354093908) coordinate (MRball-5-14) 
  -- (0.184940253674240836499667771877110173257904720,0.554820761022722509499003315631330519773714161) coordinate (MRball-1-3) 
  -- (0.169871301453858396829220458567493098187633549,0.552081729725039789694966490344352569109809033) coordinate (MRball-4-13) 
  -- (0.165379585832016306222999894629823315235230430,0.551265286106721020743332982099411050784101434) coordinate (MRball-3-10) 
  -- (0.157072976609518216626303921817509096588353042,0.549755418133313758192063726361281838059235646) coordinate (MRball-2-7) 
  -- (0.149560724861745174584373070627119355027533363,0.548389324493065640142701258966104301767622331) coordinate (MRball-3-11) 
  -- (0.136503731336895777231254597514113453378959755,0.546014925347583108925018390056453813515839022) coordinate (MRball-1-4) 
  -- (0.125515833485322216066311203196628425600798814,0.543901945103062936287348547185389844270128195) coordinate (MRball-3-13) 
  -- (0.120659574019568793307956978616344173982762422,0.542968083088059569885806403773548782922430901) coordinate (MRball-2-9) 
  -- (0.116165096082484038013270319600327023491523940,0.542103781718258844061928158134859442960445055) coordinate (MRball-3-14) 
  -- (0.108110990696686620977138617198381853614234808,0.540554953483433104885693085991909268071174038) coordinate (MRball-1-5) 
  -- (0.097922943635727906000917321624611685051544611,0.538576189996503483005045268935364267783495359) coordinate (MRball-2-11) 
  -- (0.089489708782420910714922232007204040411169487,0.536938252694525464289533392043224242467016921) coordinate (MRball-1-6) 
  -- (0.082393475584093083941302144264182015272110546,0.535557591296605045618463937717183099268718551) coordinate (MRball-2-13) 
  -- (0.076339972300858777526037931160350884512015919,0.534379806106011442682265518122456191584111434) coordinate (MRball-1-7) 
  -- (0.066559541771151917868988631820682912535759260,0.532476334169215342951909054565463300286074081) coordinate (MRball-1-8) 
  -- (0.059000570039233679981638849661909854835746904,0.531005130353103119834749646957188693521722138) coordinate (MRball-1-9) 
  -- (0.052983399800209923861174807702020491503299976,0.529833998002099238611748077020204915032999759) coordinate (MRball-1-10) 
  -- (0.048079966633840050287456918209486355268931479,0.528879632972240553162026100304349907958246270) coordinate (MRball-1-11) 
  -- (0.044007245916882955225933556839328067970615654,0.528086951002595462711202682071936815647387846) coordinate (MRball-1-12) 
  -- (0.040570621049802346434126967216038859780620569,0.527418073647430503643650573808505177148067401) coordinate (MRball-1-13) 
  -- (0.037631864095966442078793187447384117198497402,0.526846097343530189103104624263377640778963633) coordinate (MRball-1-14) 
  -- (0.000000000000000000000000000000000000000000000,0.519521730308756884400330651529448941935548067) coordinate (MRball-0-1) 
  -- (-0.035090092473999889877078045312677183342553563,0.491261294635998458279092634377480566795749883) coordinate (MRball--1-14) 
  -- (-0.037631864095966442078793187447384117198497402,0.489214233247563747024311436815993523580466231) coordinate (MRball--1-13) 
  -- (-0.040570621049802346434126967216038859780620569,0.486847452597628157209523606592466317367446831) coordinate (MRball--1-12) 
  -- (-0.044007245916882955225933556839328067970615654,0.484079705085712507485269125232608747676772192) coordinate (MRball--1-11) 
  -- (-0.048079966633840050287456918209486355268931479,0.480799666338400502874569182094863552689314791) coordinate (MRball--1-10) 
  -- (-0.052983399800209923861174807702020491503299976,0.476850598201889314750573269318184423529699783) coordinate (MRball--1-9) 
  -- (-0.059000570039233679981638849661909854835746904,0.472004560313869439853110797295278838685975234) coordinate (MRball--1-8) 
  -- (-0.066559541771151917868988631820682912535759260,0.465916792398063425082920422744780387750314821) coordinate (MRball--1-7) 
  -- (-0.071115057432706890291174417736658770574503586,0.462247873312594786892633715288282008734273309) coordinate (MRball--2-13) 
  -- (-0.076339972300858777526037931160350884512015919,0.458039833805152665156227586962105307072095515) coordinate (MRball--1-6) 
  -- (-0.082393475584093083941302144264182015272110546,0.453164115712511961677161793453001083996608005) coordinate (MRball--2-11) 
  -- (-0.089489708782420910714922232007204040411169487,0.447448543912104553574611160036020202055847434) coordinate (MRball--1-5) 
  -- (-0.094940639051445393996328810644797268588074729,0.443056315573411838649534449675720586744348736) coordinate (MRball--3-14) 
  -- (-0.097922943635727906000917321624611685051544611,0.440653246360775577004127947310752582731950748) coordinate (MRball--2-9) 
  -- (-0.101098686672216864061642475222283052861061363,0.438094308912939744267117392629893229064599238) coordinate (MRball--3-13) 
  -- (-0.108110990696686620977138617198381853614234808,0.432443962786746483908554468793527414456939230) coordinate (MRball--1-4) 
  -- (-0.116165096082484038013270319600327023491523940,0.425938685635774806048657838534532419468921115) coordinate (MRball--3-11) 
  -- (-0.120659574019568793307956978616344173982762422,0.422308509068490776577849425157204608939668478) coordinate (MRball--2-7) 
  -- (-0.125515833485322216066311203196628425600798814,0.418386111617740720221037343988761418669329380) coordinate (MRball--3-10) 
  -- (-0.128093563368515089035041645277317785938442499,0.416304080947674039363885347151282804299938121) coordinate (MRball--4-13) 
  -- (-0.136503731336895777231254597514113453378959755,0.409511194010687331693763792542340360136879266) coordinate (MRball--1-3) 
  -- (-0.144049227661203966205414378937943452703860092,0.403337837451371105375160261026241667570808257) coordinate (MRball--5-14) 
  -- (-0.146067769214823061761744470380884913627140448,0.401686365340763419844797293547433512474636233) coordinate (MRball--4-11) 
  -- (-0.149560724861745174584373070627119355027533363,0.398828599631320465558328188338984946740088968) coordinate (MRball--3-8) 
  -- (-0.152477715445555076182924212261297743413857326,0.396442060158443198075602951879374132876029047) coordinate (MRball--5-13) 
  -- (-0.157072976609518216626303921817509096588353042,0.392682441523795541565759804543772741470882604) coordinate (MRball--2-5) 
  -- (-0.161953697792610330222126410877006635878796915,0.388688874702264792533103386104815926109112596) coordinate (MRball--5-12) 
  -- (-0.165379585832016306222999894629823315235230430,0.385885700274704714520333087469587735548871004) coordinate (MRball--3-7) 
  -- (-0.169871301453858396829220458567493098187633549,0.382210428271181392865746031776859470922175484) coordinate (MRball--4-9) 
  -- (-0.172685386309666835350700810148574056079113147,0.379907849881267037771541782326862923374048923) coordinate (MRball--5-11) 
  -- (-0.174613819187903961747297383354181770643524734,0.378329941573791917119144330600727169727636924) coordinate (MRball--6-13) 
  -- (-0.184940253674240836499667771877110173257904720,0.369880507348481672999335543754220346515809441) coordinate (MRball--1-2) 
  -- (-0.194621195770247421632909960783796973584797273,0.361439363573316640175404212884194379514623506) coordinate (MRball--7-13) 
  -- (-0.196334090360518938488608468761712871398339567,0.359945832327618053895782192729806930896955872) coordinate (MRball--6-11) 
  -- (-0.198783428320450642178470371382841568340747824,0.357810170976811155921246668489114823013346083) coordinate (MRball--5-9) 
  -- (-0.202574206783901415563007042648465288302042237,0.354504861871827477235262324634814254528573914) coordinate (MRball--4-7) 
  -- (-0.205371645987043651216263463943232009089706642,0.352065678834931973513594509616969158439497100) coordinate (MRball--7-12) 
  -- (-0.209224005035894521318127551680239518774053791,0.348706675059824202196879252800399197956756318) coordinate (MRball--3-5) 
  -- (-0.212715326419959361740273861663386252969010023,0.345662405432433962827945025203002661074641288) coordinate (MRball--8-13) 
  -- (-0.214866614188422494962075064664193783599356788,0.343786582701475991939320103462710053758970861) coordinate (MRball--5-8) 
  -- (-0.217379134181769143451868376293037015191105452,0.341595782285637225424364591317629595300308567) coordinate (MRball--7-11) 
  -- (-0.218800534668312700596715653452027268870284337,0.340356387261819756483779905369820196020442302) coordinate (MRball--9-14) 
  -- (-0.223925249650749114702336361152774224996743997,0.335887874476123672053504541729161337495115996) coordinate (MRball--2-3) 
  -- (-0.229289606886223787241152398177207517570076433,0.331196098835656581570553464033744192045665959) coordinate (MRball--9-13) 
  -- (-0.230869812976842507349876598587420074671929908,0.329814018538346439071252283696314392388471297) coordinate (MRball--7-10) 
  -- (-0.233769760902938209935340620746117264224894714,0.327277665264113493909476869044564169914852599) coordinate (MRball--5-7) 
  -- (-0.236367641088999059398296436991696777155032625,0.325005506497373706672657600863583068588169860) coordinate (MRball--8-11) 
  -- (-0.240828174829680712121781155095334827143215921,0.321104233106240949495708206793779769524287895) coordinate (MRball--3-4) 
  -- (-0.244519486431325909209436967856386435997751507,0.317875332360723681972268058213302366797076959) coordinate (MRball--10-13) 
  -- (-0.246136346284567169714240282372281374202016367,0.316461016651586361061166077335790338259735329) coordinate (MRball--7-9) 
  -- (-0.247624886274565107053343486191879631555371556,0.315158946167628318067891709698755894706836526) coordinate (MRball--11-14) 
  -- (-0.250273617409824594072278462702309938466617296,0.312842021762280742590348078377887423083271620) coordinate (MRball--4-5) 
  -- (-0.253588920139773620940215038598087916825170695,0.309942013504167758926929491619885231675208628) coordinate (MRball--9-11) 
  -- (-0.256305079965498994192801513212643948670368301,0.307566095958598793031361815855172738404441961) coordinate (MRball--5-6) 
  -- (-0.258571053035912916324463319985415868353072568,0.305583971769715264747093014528218753508176671) coordinate (MRball--11-13) 
  -- (-0.260490197507343535663943184451683859940904755,0.303905230425234124941267048526964503264388881) coordinate (MRball--6-7) 
  -- (-0.263564233040238078544563906490042764995865646,0.301216266331700661193787321702906017138132167) coordinate (MRball--7-8) 
  -- (-0.265917797878547109099414312796625673834339591,0.299157522613365497736841101896203883063632039) coordinate (MRball--8-9) 
  -- (-0.267777611771587492121391143495745269944515863,0.297530679746208324579323492773050299938350959) coordinate (MRball--9-10) 
  -- (-0.269284298979181813201161635542805370897977719,0.296212728877099994521277799097085907987775491) coordinate (MRball--10-11) 
  -- (-0.270529712661283145792286634845741114441632417,0.295123322903217977227949056195353943027235364) coordinate (MRball--11-12) 
  -- (-0.271576391316405008390858989376477173976763223,0.294207757259438759090097238491183605141493491) coordinate (MRball--12-13) 
  -- (-0.272468389261054077772558961754701097133845036,0.293427496127289006831986574197370412297986962) coordinate (MRball--13-14) 
  -- (-0.283648164276627746014311351288321559465747698,0.283648164276627746014311351288321559465747698) coordinate (MRball--1-1) 
  -- (-0.293427496127289006831986574197370412297986962,0.272468389261054077772558961754701097133845036) coordinate (MRball--14-13) 
  -- (-0.294207757259438759090097238491183605141493491,0.271576391316405008390858989376477173976763223) coordinate (MRball--13-12) 
  -- (-0.295123322903217977227949056195353943027235364,0.270529712661283145792286634845741114441632417) coordinate (MRball--12-11) 
  -- (-0.296212728877099994521277799097085907987775491,0.269284298979181813201161635542805370897977719) coordinate (MRball--11-10) 
  -- (-0.297530679746208324579323492773050299938350959,0.267777611771587492121391143495745269944515863) coordinate (MRball--10-9) 
  -- (-0.299157522613365497736841101896203883063632039,0.265917797878547109099414312796625673834339591) coordinate (MRball--9-8) 
  -- (-0.301216266331700661193787321702906017138132167,0.263564233040238078544563906490042764995865646) coordinate (MRball--8-7) 
  -- (-0.303905230425234124941267048526964503264388881,0.260490197507343535663943184451683859940904755) coordinate (MRball--7-6) 
  -- (-0.305583971769715264747093014528218753508176671,0.258571053035912916324463319985415868353072568) coordinate (MRball--13-11) 
  -- (-0.307566095958598793031361815855172738404441961,0.256305079965498994192801513212643948670368301) coordinate (MRball--6-5) 
  -- (-0.309942013504167758926929491619885231675208628,0.253588920139773620940215038598087916825170695) coordinate (MRball--11-9) 
  -- (-0.312842021762280742590348078377887423083271620,0.250273617409824594072278462702309938466617296) coordinate (MRball--5-4) 
  -- (-0.315158946167628318067891709698755894706836526,0.247624886274565107053343486191879631555371556) coordinate (MRball--14-11) 
  -- (-0.316461016651586361061166077335790338259735329,0.246136346284567169714240282372281374202016367) coordinate (MRball--9-7) 
  -- (-0.317875332360723681972268058213302366797076959,0.244519486431325909209436967856386435997751507) coordinate (MRball--13-10) 
  -- (-0.321104233106240949495708206793779769524287895,0.240828174829680712121781155095334827143215921) coordinate (MRball--4-3) 
  -- (-0.325005506497373706672657600863583068588169860,0.236367641088999059398296436991696777155032625) coordinate (MRball--11-8) 
  -- (-0.327277665264113493909476869044564169914852599,0.233769760902938209935340620746117264224894714) coordinate (MRball--7-5) 
  -- (-0.329814018538346439071252283696314392388471297,0.230869812976842507349876598587420074671929908) coordinate (MRball--10-7) 
  -- (-0.331196098835656581570553464033744192045665959,0.229289606886223787241152398177207517570076433) coordinate (MRball--13-9) 
  -- (-0.335887874476123672053504541729161337495115996,0.223925249650749114702336361152774224996743997) coordinate (MRball--3-2) 
  -- (-0.340356387261819756483779905369820196020442302,0.218800534668312700596715653452027268870284337) coordinate (MRball--14-9) 
  -- (-0.341595782285637225424364591317629595300308567,0.217379134181769143451868376293037015191105452) coordinate (MRball--11-7) 
  -- (-0.343786582701475991939320103462710053758970861,0.214866614188422494962075064664193783599356788) coordinate (MRball--8-5) 
  -- (-0.345662405432433962827945025203002661074641288,0.212715326419959361740273861663386252969010023) coordinate (MRball--13-8) 
  -- (-0.348706675059824202196879252800399197956756318,0.209224005035894521318127551680239518774053791) coordinate (MRball--5-3) 
  -- (-0.352065678834931973513594509616969158439497100,0.205371645987043651216263463943232009089706642) coordinate (MRball--12-7) 
  -- (-0.354504861871827477235262324634814254528573914,0.202574206783901415563007042648465288302042237) coordinate (MRball--7-4) 
  -- (-0.357810170976811155921246668489114823013346083,0.198783428320450642178470371382841568340747824) coordinate (MRball--9-5) 
  -- (-0.359945832327618053895782192729806930896955872,0.196334090360518938488608468761712871398339567) coordinate (MRball--11-6) 
  -- (-0.361439363573316640175404212884194379514623506,0.194621195770247421632909960783796973584797273) coordinate (MRball--13-7) 
  -- (-0.369880507348481672999335543754220346515809441,0.184940253674240836499667771877110173257904720) coordinate (MRball--2-1) 
  -- (-0.378329941573791917119144330600727169727636924,0.174613819187903961747297383354181770643524734) coordinate (MRball--13-6) 
  -- (-0.379907849881267037771541782326862923374048923,0.172685386309666835350700810148574056079113147) coordinate (MRball--11-5) 
  -- (-0.382210428271181392865746031776859470922175484,0.169871301453858396829220458567493098187633549) coordinate (MRball--9-4) 
  -- (-0.385885700274704714520333087469587735548871004,0.165379585832016306222999894629823315235230430) coordinate (MRball--7-3) 
  -- (-0.388688874702264792533103386104815926109112596,0.161953697792610330222126410877006635878796915) coordinate (MRball--12-5) 
  -- (-0.392682441523795541565759804543772741470882604,0.157072976609518216626303921817509096588353042) coordinate (MRball--5-2) 
  -- (-0.396442060158443198075602951879374132876029047,0.152477715445555076182924212261297743413857326) coordinate (MRball--13-5) 
  -- (-0.398828599631320465558328188338984946740088968,0.149560724861745174584373070627119355027533363) coordinate (MRball--8-3) 
  -- (-0.401686365340763419844797293547433512474636233,0.146067769214823061761744470380884913627140448) coordinate (MRball--11-4) 
  -- (-0.403337837451371105375160261026241667570808257,0.144049227661203966205414378937943452703860092) coordinate (MRball--14-5) 
  -- (-0.409511194010687331693763792542340360136879266,0.136503731336895777231254597514113453378959755) coordinate (MRball--3-1) 
  -- (-0.416304080947674039363885347151282804299938121,0.128093563368515089035041645277317785938442499) coordinate (MRball--13-4) 
  -- (-0.418386111617740720221037343988761418669329380,0.125515833485322216066311203196628425600798814) coordinate (MRball--10-3) 
  -- (-0.422308509068490776577849425157204608939668478,0.120659574019568793307956978616344173982762422) coordinate (MRball--7-2) 
  -- (-0.425938685635774806048657838534532419468921115,0.116165096082484038013270319600327023491523940) coordinate (MRball--11-3) 
  -- (-0.432443962786746483908554468793527414456939230,0.108110990696686620977138617198381853614234808) coordinate (MRball--4-1) 
  -- (-0.438094308912939744267117392629893229064599238,0.101098686672216864061642475222283052861061363) coordinate (MRball--13-3) 
  -- (-0.440653246360775577004127947310752582731950748,0.097922943635727906000917321624611685051544611) coordinate (MRball--9-2) 
  -- (-0.443056315573411838649534449675720586744348736,0.094940639051445393996328810644797268588074729) coordinate (MRball--14-3) 
  -- (-0.447448543912104553574611160036020202055847434,0.089489708782420910714922232007204040411169487) coordinate (MRball--5-1) 
  -- (-0.453164115712511961677161793453001083996608005,0.082393475584093083941302144264182015272110546) coordinate (MRball--11-2) 
  -- (-0.458039833805152665156227586962105307072095515,0.076339972300858777526037931160350884512015919) coordinate (MRball--6-1) 
  -- (-0.462247873312594786892633715288282008734273309,0.071115057432706890291174417736658770574503586) coordinate (MRball--13-2) 
  -- (-0.465916792398063425082920422744780387750314821,0.066559541771151917868988631820682912535759260) coordinate (MRball--7-1) 
  -- (-0.472004560313869439853110797295278838685975234,0.059000570039233679981638849661909854835746904) coordinate (MRball--8-1) 
  -- (-0.476850598201889314750573269318184423529699783,0.052983399800209923861174807702020491503299976) coordinate (MRball--9-1) 
  -- (-0.480799666338400502874569182094863552689314791,0.048079966633840050287456918209486355268931479) coordinate (MRball--10-1) 
  -- (-0.484079705085712507485269125232608747676772192,0.044007245916882955225933556839328067970615654) coordinate (MRball--11-1) 
  -- (-0.486847452597628157209523606592466317367446831,0.040570621049802346434126967216038859780620569) coordinate (MRball--12-1) 
  -- (-0.489214233247563747024311436815993523580466231,0.037631864095966442078793187447384117198497402) coordinate (MRball--13-1) 
  -- (-0.491261294635998458279092634377480566795749883,0.035090092473999889877078045312677183342553563) coordinate (MRball--14-1) 
  -- (-0.519521730308756884400330651529448941935548067,0.000000000000000000000000000000000000000000000) coordinate (MRball--1-0) 
  -- (-0.526846097343530189103104624263377640778963633,-0.037631864095966442078793187447384117198497402) coordinate (MRball--14--1) 
  -- (-0.527418073647430503643650573808505177148067401,-0.040570621049802346434126967216038859780620569) coordinate (MRball--13--1) 
  -- (-0.528086951002595462711202682071936815647387846,-0.044007245916882955225933556839328067970615654) coordinate (MRball--12--1) 
  -- (-0.528879632972240553162026100304349907958246270,-0.048079966633840050287456918209486355268931479) coordinate (MRball--11--1) 
  -- (-0.529833998002099238611748077020204915032999759,-0.052983399800209923861174807702020491503299976) coordinate (MRball--10--1) 
  -- (-0.531005130353103119834749646957188693521722138,-0.059000570039233679981638849661909854835746904) coordinate (MRball--9--1) 
  -- (-0.532476334169215342951909054565463300286074081,-0.066559541771151917868988631820682912535759260) coordinate (MRball--8--1) 
  -- (-0.534379806106011442682265518122456191584111434,-0.076339972300858777526037931160350884512015919) coordinate (MRball--7--1) 
  -- (-0.535557591296605045618463937717183099268718551,-0.082393475584093083941302144264182015272110546) coordinate (MRball--13--2) 
  -- (-0.536938252694525464289533392043224242467016921,-0.089489708782420910714922232007204040411169487) coordinate (MRball--6--1) 
  -- (-0.538576189996503483005045268935364267783495359,-0.097922943635727906000917321624611685051544611) coordinate (MRball--11--2) 
  -- (-0.540554953483433104885693085991909268071174038,-0.108110990696686620977138617198381853614234808) coordinate (MRball--5--1) 
  -- (-0.542103781718258844061928158134859442960445055,-0.116165096082484038013270319600327023491523940) coordinate (MRball--14--3) 
  -- (-0.542968083088059569885806403773548782922430901,-0.120659574019568793307956978616344173982762422) coordinate (MRball--9--2) 
  -- (-0.543901945103062936287348547185389844270128195,-0.125515833485322216066311203196628425600798814) coordinate (MRball--13--3) 
  -- (-0.546014925347583108925018390056453813515839022,-0.136503731336895777231254597514113453378959755) coordinate (MRball--4--1) 
  -- (-0.548389324493065640142701258966104301767622331,-0.149560724861745174584373070627119355027533363) coordinate (MRball--11--3) 
  -- (-0.549755418133313758192063726361281838059235646,-0.157072976609518216626303921817509096588353042) coordinate (MRball--7--2) 
  -- (-0.551265286106721020743332982099411050784101434,-0.165379585832016306222999894629823315235230430) coordinate (MRball--10--3) 
  -- (-0.552081729725039789694966490344352569109809033,-0.169871301453858396829220458567493098187633549) coordinate (MRball--13--4) 
  -- (-0.554820761022722509499003315631330519773714161,-0.184940253674240836499667771877110173257904720) coordinate (MRball--3--1) 
  -- (-0.556593599297261798099717039871956391354093908,-0.198783428320450642178470371382841568340747824) coordinate (MRball--14--5) 
  -- (-0.557079068655728892798269367283279542830616151,-0.202574206783901415563007042648465288302042237) coordinate (MRball--11--4) 
  -- (-0.557930680095718723515006804480638716730810108,-0.209224005035894521318127551680239518774053791) coordinate (MRball--8--3) 
  -- (-0.558653196889898486901395168126903837358327650,-0.214866614188422494962075064664193783599356788) coordinate (MRball--13--5) 
  -- (-0.559813124126872786755840902881935562491859993,-0.223925249650749114702336361152774224996743997) coordinate (MRball--5--2) 
  -- (-0.561047426167051703844817489790681434139747313,-0.233769760902938209935340620746117264224894714) coordinate (MRball--12--5) 
  -- (-0.561932407935921661617489361889114596667503816,-0.240828174829680712121781155095334827143215921) coordinate (MRball--7--3) 
  -- (-0.563115639172105336662626541080197361549888915,-0.250273617409824594072278462702309938466617296) coordinate (MRball--9--4) 
  -- (-0.563871175924097787224163329067816687074810262,-0.256305079965498994192801513212643948670368301) coordinate (MRball--11--5) 
  -- (-0.564395427932577660605210232978648363205293637,-0.260490197507343535663943184451683859940904755) coordinate (MRball--13--6) 
  -- (-0.567296328553255492028622702576643118931495396,-0.283648164276627746014311351288321559465747698) coordinate (MRball--2--1) 
  -- (-0.564395427932577660605210232978648363205293637,-0.303905230425234124941267048526964503264388881) coordinate (MRball--13--7) 
  -- (-0.563871175924097787224163329067816687074810262,-0.307566095958598793031361815855172738404441961) coordinate (MRball--11--6) 
  -- (-0.563115639172105336662626541080197361549888915,-0.312842021762280742590348078377887423083271620) coordinate (MRball--9--5) 
  -- (-0.561932407935921661617489361889114596667503816,-0.321104233106240949495708206793779769524287895) coordinate (MRball--7--4) 
  -- (-0.561047426167051703844817489790681434139747313,-0.327277665264113493909476869044564169914852599) coordinate (MRball--12--7) 
  -- (-0.559813124126872786755840902881935562491859993,-0.335887874476123672053504541729161337495115996) coordinate (MRball--5--3) 
  -- (-0.558653196889898486901395168126903837358327650,-0.343786582701475991939320103462710053758970861) coordinate (MRball--13--8) 
  -- (-0.557930680095718723515006804480638716730810108,-0.348706675059824202196879252800399197956756318) coordinate (MRball--8--5) 
  -- (-0.557079068655728892798269367283279542830616151,-0.354504861871827477235262324634814254528573914) coordinate (MRball--11--7) 
  -- (-0.556593599297261798099717039871956391354093908,-0.357810170976811155921246668489114823013346083) coordinate (MRball--14--9) 
  -- (-0.554820761022722509499003315631330519773714161,-0.369880507348481672999335543754220346515809441) coordinate (MRball--3--2) 
  -- (-0.552081729725039789694966490344352569109809033,-0.382210428271181392865746031776859470922175484) coordinate (MRball--13--9) 
  -- (-0.551265286106721020743332982099411050784101434,-0.385885700274704714520333087469587735548871004) coordinate (MRball--10--7) 
  -- (-0.549755418133313758192063726361281838059235646,-0.392682441523795541565759804543772741470882604) coordinate (MRball--7--5) 
  -- (-0.548389324493065640142701258966104301767622331,-0.398828599631320465558328188338984946740088968) coordinate (MRball--11--8) 
  -- (-0.546014925347583108925018390056453813515839022,-0.409511194010687331693763792542340360136879266) coordinate (MRball--4--3) 
  -- (-0.543901945103062936287348547185389844270128195,-0.418386111617740720221037343988761418669329380) coordinate (MRball--13--10) 
  -- (-0.542968083088059569885806403773548782922430901,-0.422308509068490776577849425157204608939668478) coordinate (MRball--9--7) 
  -- (-0.542103781718258844061928158134859442960445055,-0.425938685635774806048657838534532419468921115) coordinate (MRball--14--11) 
  -- (-0.540554953483433104885693085991909268071174038,-0.432443962786746483908554468793527414456939230) coordinate (MRball--5--4) 
  -- (-0.538576189996503483005045268935364267783495359,-0.440653246360775577004127947310752582731950748) coordinate (MRball--11--9) 
  -- (-0.536938252694525464289533392043224242467016921,-0.447448543912104553574611160036020202055847434) coordinate (MRball--6--5) 
  -- (-0.535557591296605045618463937717183099268718551,-0.453164115712511961677161793453001083996608005) coordinate (MRball--13--11) 
  -- (-0.534379806106011442682265518122456191584111434,-0.458039833805152665156227586962105307072095515) coordinate (MRball--7--6) 
  -- (-0.532476334169215342951909054565463300286074081,-0.465916792398063425082920422744780387750314821) coordinate (MRball--8--7) 
  -- (-0.531005130353103119834749646957188693521722138,-0.472004560313869439853110797295278838685975234) coordinate (MRball--9--8) 
  -- (-0.529833998002099238611748077020204915032999759,-0.476850598201889314750573269318184423529699783) coordinate (MRball--10--9) 
  -- (-0.528879632972240553162026100304349907958246270,-0.480799666338400502874569182094863552689314791) coordinate (MRball--11--10) 
  -- (-0.528086951002595462711202682071936815647387846,-0.484079705085712507485269125232608747676772192) coordinate (MRball--12--11) 
  -- (-0.527418073647430503643650573808505177148067401,-0.486847452597628157209523606592466317367446831) coordinate (MRball--13--12) 
  -- (-0.526846097343530189103104624263377640778963633,-0.489214233247563747024311436815993523580466231) coordinate (MRball--14--13) 
  -- (-0.519521730308756884400330651529448941935548067,-0.519521730308756884400330651529448941935548067) coordinate (MRball--1--1) 
  -- (-0.489214233247563747024311436815993523580466231,-0.526846097343530189103104624263377640778963633) coordinate (MRball--13--14) 
  -- (-0.486847452597628157209523606592466317367446831,-0.527418073647430503643650573808505177148067401) coordinate (MRball--12--13) 
  -- (-0.484079705085712507485269125232608747676772192,-0.528086951002595462711202682071936815647387846) coordinate (MRball--11--12) 
  -- (-0.480799666338400502874569182094863552689314791,-0.528879632972240553162026100304349907958246270) coordinate (MRball--10--11) 
  -- (-0.476850598201889314750573269318184423529699783,-0.529833998002099238611748077020204915032999759) coordinate (MRball--9--10) 
  -- (-0.472004560313869439853110797295278838685975234,-0.531005130353103119834749646957188693521722138) coordinate (MRball--8--9) 
  -- (-0.465916792398063425082920422744780387750314821,-0.532476334169215342951909054565463300286074081) coordinate (MRball--7--8) 
  -- (-0.458039833805152665156227586962105307072095515,-0.534379806106011442682265518122456191584111434) coordinate (MRball--6--7) 
  -- (-0.453164115712511961677161793453001083996608005,-0.535557591296605045618463937717183099268718551) coordinate (MRball--11--13) 
  -- (-0.447448543912104553574611160036020202055847434,-0.536938252694525464289533392043224242467016921) coordinate (MRball--5--6) 
  -- (-0.440653246360775577004127947310752582731950748,-0.538576189996503483005045268935364267783495359) coordinate (MRball--9--11) 
  -- (-0.432443962786746483908554468793527414456939230,-0.540554953483433104885693085991909268071174038) coordinate (MRball--4--5) 
  -- (-0.425938685635774806048657838534532419468921115,-0.542103781718258844061928158134859442960445055) coordinate (MRball--11--14) 
  -- (-0.422308509068490776577849425157204608939668478,-0.542968083088059569885806403773548782922430901) coordinate (MRball--7--9) 
  -- (-0.418386111617740720221037343988761418669329380,-0.543901945103062936287348547185389844270128195) coordinate (MRball--10--13) 
  -- (-0.409511194010687331693763792542340360136879266,-0.546014925347583108925018390056453813515839022) coordinate (MRball--3--4) 
  -- (-0.398828599631320465558328188338984946740088968,-0.548389324493065640142701258966104301767622331) coordinate (MRball--8--11) 
  -- (-0.392682441523795541565759804543772741470882604,-0.549755418133313758192063726361281838059235646) coordinate (MRball--5--7) 
  -- (-0.385885700274704714520333087469587735548871004,-0.551265286106721020743332982099411050784101434) coordinate (MRball--7--10) 
  -- (-0.382210428271181392865746031776859470922175484,-0.552081729725039789694966490344352569109809033) coordinate (MRball--9--13) 
  -- (-0.369880507348481672999335543754220346515809441,-0.554820761022722509499003315631330519773714161) coordinate (MRball--2--3) 
  -- (-0.357810170976811155921246668489114823013346083,-0.556593599297261798099717039871956391354093908) coordinate (MRball--9--14) 
  -- (-0.354504861871827477235262324634814254528573914,-0.557079068655728892798269367283279542830616151) coordinate (MRball--7--11) 
  -- (-0.348706675059824202196879252800399197956756318,-0.557930680095718723515006804480638716730810108) coordinate (MRball--5--8) 
  -- (-0.343786582701475991939320103462710053758970861,-0.558653196889898486901395168126903837358327650) coordinate (MRball--8--13) 
  -- (-0.335887874476123672053504541729161337495115996,-0.559813124126872786755840902881935562491859993) coordinate (MRball--3--5) 
  -- (-0.327277665264113493909476869044564169914852599,-0.561047426167051703844817489790681434139747313) coordinate (MRball--7--12) 
  -- (-0.321104233106240949495708206793779769524287895,-0.561932407935921661617489361889114596667503816) coordinate (MRball--4--7) 
  -- (-0.312842021762280742590348078377887423083271620,-0.563115639172105336662626541080197361549888915) coordinate (MRball--5--9) 
  -- (-0.307566095958598793031361815855172738404441961,-0.563871175924097787224163329067816687074810262) coordinate (MRball--6--11) 
  -- (-0.303905230425234124941267048526964503264388881,-0.564395427932577660605210232978648363205293637) coordinate (MRball--7--13) 
  -- (-0.283648164276627746014311351288321559465747698,-0.567296328553255492028622702576643118931495396) coordinate (MRball--1--2) 
  -- (-0.260490197507343535663943184451683859940904755,-0.564395427932577660605210232978648363205293637) coordinate (MRball--6--13) 
  -- (-0.256305079965498994192801513212643948670368301,-0.563871175924097787224163329067816687074810262) coordinate (MRball--5--11) 
  -- (-0.250273617409824594072278462702309938466617296,-0.563115639172105336662626541080197361549888915) coordinate (MRball--4--9) 
  -- (-0.240828174829680712121781155095334827143215921,-0.561932407935921661617489361889114596667503816) coordinate (MRball--3--7) 
  -- (-0.233769760902938209935340620746117264224894714,-0.561047426167051703844817489790681434139747313) coordinate (MRball--5--12) 
  -- (-0.223925249650749114702336361152774224996743997,-0.559813124126872786755840902881935562491859993) coordinate (MRball--2--5) 
  -- (-0.214866614188422494962075064664193783599356788,-0.558653196889898486901395168126903837358327650) coordinate (MRball--5--13) 
  -- (-0.209224005035894521318127551680239518774053791,-0.557930680095718723515006804480638716730810108) coordinate (MRball--3--8) 
  -- (-0.202574206783901415563007042648465288302042237,-0.557079068655728892798269367283279542830616151) coordinate (MRball--4--11) 
  -- (-0.198783428320450642178470371382841568340747824,-0.556593599297261798099717039871956391354093908) coordinate (MRball--5--14) 
  -- (-0.184940253674240836499667771877110173257904720,-0.554820761022722509499003315631330519773714161) coordinate (MRball--1--3) 
  -- (-0.169871301453858396829220458567493098187633549,-0.552081729725039789694966490344352569109809033) coordinate (MRball--4--13) 
  -- (-0.165379585832016306222999894629823315235230430,-0.551265286106721020743332982099411050784101434) coordinate (MRball--3--10) 
  -- (-0.157072976609518216626303921817509096588353042,-0.549755418133313758192063726361281838059235646) coordinate (MRball--2--7) 
  -- (-0.149560724861745174584373070627119355027533363,-0.548389324493065640142701258966104301767622331) coordinate (MRball--3--11) 
  -- (-0.136503731336895777231254597514113453378959755,-0.546014925347583108925018390056453813515839022) coordinate (MRball--1--4) 
  -- (-0.125515833485322216066311203196628425600798814,-0.543901945103062936287348547185389844270128195) coordinate (MRball--3--13) 
  -- (-0.120659574019568793307956978616344173982762422,-0.542968083088059569885806403773548782922430901) coordinate (MRball--2--9) 
  -- (-0.116165096082484038013270319600327023491523940,-0.542103781718258844061928158134859442960445055) coordinate (MRball--3--14) 
  -- (-0.108110990696686620977138617198381853614234808,-0.540554953483433104885693085991909268071174038) coordinate (MRball--1--5) 
  -- (-0.097922943635727906000917321624611685051544611,-0.538576189996503483005045268935364267783495359) coordinate (MRball--2--11) 
  -- (-0.089489708782420910714922232007204040411169487,-0.536938252694525464289533392043224242467016921) coordinate (MRball--1--6) 
  -- (-0.082393475584093083941302144264182015272110546,-0.535557591296605045618463937717183099268718551) coordinate (MRball--2--13) 
  -- (-0.076339972300858777526037931160350884512015919,-0.534379806106011442682265518122456191584111434) coordinate (MRball--1--7) 
  -- (-0.066559541771151917868988631820682912535759260,-0.532476334169215342951909054565463300286074081) coordinate (MRball--1--8) 
  -- (-0.059000570039233679981638849661909854835746904,-0.531005130353103119834749646957188693521722138) coordinate (MRball--1--9) 
  -- (-0.052983399800209923861174807702020491503299976,-0.529833998002099238611748077020204915032999759) coordinate (MRball--1--10) 
  -- (-0.048079966633840050287456918209486355268931479,-0.528879632972240553162026100304349907958246270) coordinate (MRball--1--11) 
  -- (-0.044007245916882955225933556839328067970615654,-0.528086951002595462711202682071936815647387846) coordinate (MRball--1--12) 
  -- (-0.040570621049802346434126967216038859780620569,-0.527418073647430503643650573808505177148067401) coordinate (MRball--1--13) 
  -- (-0.037631864095966442078793187447384117198497402,-0.526846097343530189103104624263377640778963633) coordinate (MRball--1--14) 
  -- (0.000000000000000000000000000000000000000000000,-0.519521730308756884400330651529448941935548067) coordinate (MRball-0--1) 
  -- (0.035090092473999889877078045312677183342553563,-0.491261294635998458279092634377480566795749883) coordinate (MRball-1--14) 
  -- (0.037631864095966442078793187447384117198497402,-0.489214233247563747024311436815993523580466231) coordinate (MRball-1--13) 
  -- (0.040570621049802346434126967216038859780620569,-0.486847452597628157209523606592466317367446831) coordinate (MRball-1--12) 
  -- (0.044007245916882955225933556839328067970615654,-0.484079705085712507485269125232608747676772192) coordinate (MRball-1--11) 
  -- (0.048079966633840050287456918209486355268931479,-0.480799666338400502874569182094863552689314791) coordinate (MRball-1--10) 
  -- (0.052983399800209923861174807702020491503299976,-0.476850598201889314750573269318184423529699783) coordinate (MRball-1--9) 
  -- (0.059000570039233679981638849661909854835746904,-0.472004560313869439853110797295278838685975234) coordinate (MRball-1--8) 
  -- (0.066559541771151917868988631820682912535759260,-0.465916792398063425082920422744780387750314821) coordinate (MRball-1--7) 
  -- (0.071115057432706890291174417736658770574503586,-0.462247873312594786892633715288282008734273309) coordinate (MRball-2--13) 
  -- (0.076339972300858777526037931160350884512015919,-0.458039833805152665156227586962105307072095515) coordinate (MRball-1--6) 
  -- (0.082393475584093083941302144264182015272110546,-0.453164115712511961677161793453001083996608005) coordinate (MRball-2--11) 
  -- (0.089489708782420910714922232007204040411169487,-0.447448543912104553574611160036020202055847434) coordinate (MRball-1--5) 
  -- (0.094940639051445393996328810644797268588074729,-0.443056315573411838649534449675720586744348736) coordinate (MRball-3--14) 
  -- (0.097922943635727906000917321624611685051544611,-0.440653246360775577004127947310752582731950748) coordinate (MRball-2--9) 
  -- (0.101098686672216864061642475222283052861061363,-0.438094308912939744267117392629893229064599238) coordinate (MRball-3--13) 
  -- (0.108110990696686620977138617198381853614234808,-0.432443962786746483908554468793527414456939230) coordinate (MRball-1--4) 
  -- (0.116165096082484038013270319600327023491523940,-0.425938685635774806048657838534532419468921115) coordinate (MRball-3--11) 
  -- (0.120659574019568793307956978616344173982762422,-0.422308509068490776577849425157204608939668478) coordinate (MRball-2--7) 
  -- (0.125515833485322216066311203196628425600798814,-0.418386111617740720221037343988761418669329380) coordinate (MRball-3--10) 
  -- (0.128093563368515089035041645277317785938442499,-0.416304080947674039363885347151282804299938121) coordinate (MRball-4--13) 
  -- (0.136503731336895777231254597514113453378959755,-0.409511194010687331693763792542340360136879266) coordinate (MRball-1--3) 
  -- (0.144049227661203966205414378937943452703860092,-0.403337837451371105375160261026241667570808257) coordinate (MRball-5--14) 
  -- (0.146067769214823061761744470380884913627140448,-0.401686365340763419844797293547433512474636233) coordinate (MRball-4--11) 
  -- (0.149560724861745174584373070627119355027533363,-0.398828599631320465558328188338984946740088968) coordinate (MRball-3--8) 
  -- (0.152477715445555076182924212261297743413857326,-0.396442060158443198075602951879374132876029047) coordinate (MRball-5--13) 
  -- (0.157072976609518216626303921817509096588353042,-0.392682441523795541565759804543772741470882604) coordinate (MRball-2--5) 
  -- (0.161953697792610330222126410877006635878796915,-0.388688874702264792533103386104815926109112596) coordinate (MRball-5--12) 
  -- (0.165379585832016306222999894629823315235230430,-0.385885700274704714520333087469587735548871004) coordinate (MRball-3--7) 
  -- (0.169871301453858396829220458567493098187633549,-0.382210428271181392865746031776859470922175484) coordinate (MRball-4--9) 
  -- (0.172685386309666835350700810148574056079113147,-0.379907849881267037771541782326862923374048923) coordinate (MRball-5--11) 
  -- (0.174613819187903961747297383354181770643524734,-0.378329941573791917119144330600727169727636924) coordinate (MRball-6--13) 
  -- (0.184940253674240836499667771877110173257904720,-0.369880507348481672999335543754220346515809441) coordinate (MRball-1--2) 
  -- (0.194621195770247421632909960783796973584797273,-0.361439363573316640175404212884194379514623506) coordinate (MRball-7--13) 
  -- (0.196334090360518938488608468761712871398339567,-0.359945832327618053895782192729806930896955872) coordinate (MRball-6--11) 
  -- (0.198783428320450642178470371382841568340747824,-0.357810170976811155921246668489114823013346083) coordinate (MRball-5--9) 
  -- (0.202574206783901415563007042648465288302042237,-0.354504861871827477235262324634814254528573914) coordinate (MRball-4--7) 
  -- (0.205371645987043651216263463943232009089706642,-0.352065678834931973513594509616969158439497100) coordinate (MRball-7--12) 
  -- (0.209224005035894521318127551680239518774053791,-0.348706675059824202196879252800399197956756318) coordinate (MRball-3--5) 
  -- (0.212715326419959361740273861663386252969010023,-0.345662405432433962827945025203002661074641288) coordinate (MRball-8--13) 
  -- (0.214866614188422494962075064664193783599356788,-0.343786582701475991939320103462710053758970861) coordinate (MRball-5--8) 
  -- (0.217379134181769143451868376293037015191105452,-0.341595782285637225424364591317629595300308567) coordinate (MRball-7--11) 
  -- (0.218800534668312700596715653452027268870284337,-0.340356387261819756483779905369820196020442302) coordinate (MRball-9--14) 
  -- (0.223925249650749114702336361152774224996743997,-0.335887874476123672053504541729161337495115996) coordinate (MRball-2--3) 
  -- (0.229289606886223787241152398177207517570076433,-0.331196098835656581570553464033744192045665959) coordinate (MRball-9--13) 
  -- (0.230869812976842507349876598587420074671929908,-0.329814018538346439071252283696314392388471297) coordinate (MRball-7--10) 
  -- (0.233769760902938209935340620746117264224894714,-0.327277665264113493909476869044564169914852599) coordinate (MRball-5--7) 
  -- (0.236367641088999059398296436991696777155032625,-0.325005506497373706672657600863583068588169860) coordinate (MRball-8--11) 
  -- (0.240828174829680712121781155095334827143215921,-0.321104233106240949495708206793779769524287895) coordinate (MRball-3--4) 
  -- (0.244519486431325909209436967856386435997751507,-0.317875332360723681972268058213302366797076959) coordinate (MRball-10--13) 
  -- (0.246136346284567169714240282372281374202016367,-0.316461016651586361061166077335790338259735329) coordinate (MRball-7--9) 
  -- (0.247624886274565107053343486191879631555371556,-0.315158946167628318067891709698755894706836526) coordinate (MRball-11--14) 
  -- (0.250273617409824594072278462702309938466617296,-0.312842021762280742590348078377887423083271620) coordinate (MRball-4--5) 
  -- (0.253588920139773620940215038598087916825170695,-0.309942013504167758926929491619885231675208628) coordinate (MRball-9--11) 
  -- (0.256305079965498994192801513212643948670368301,-0.307566095958598793031361815855172738404441961) coordinate (MRball-5--6) 
  -- (0.258571053035912916324463319985415868353072568,-0.305583971769715264747093014528218753508176671) coordinate (MRball-11--13) 
  -- (0.260490197507343535663943184451683859940904755,-0.303905230425234124941267048526964503264388881) coordinate (MRball-6--7) 
  -- (0.263564233040238078544563906490042764995865646,-0.301216266331700661193787321702906017138132167) coordinate (MRball-7--8) 
  -- (0.265917797878547109099414312796625673834339591,-0.299157522613365497736841101896203883063632039) coordinate (MRball-8--9) 
  -- (0.267777611771587492121391143495745269944515863,-0.297530679746208324579323492773050299938350959) coordinate (MRball-9--10) 
  -- (0.269284298979181813201161635542805370897977719,-0.296212728877099994521277799097085907987775491) coordinate (MRball-10--11) 
  -- (0.270529712661283145792286634845741114441632417,-0.295123322903217977227949056195353943027235364) coordinate (MRball-11--12) 
  -- (0.271576391316405008390858989376477173976763223,-0.294207757259438759090097238491183605141493491) coordinate (MRball-12--13) 
  -- (0.272468389261054077772558961754701097133845036,-0.293427496127289006831986574197370412297986962) coordinate (MRball-13--14) 
  -- (0.283648164276627746014311351288321559465747698,-0.283648164276627746014311351288321559465747698) coordinate (MRball-1--1) 
  -- (0.293427496127289006831986574197370412297986962,-0.272468389261054077772558961754701097133845036) coordinate (MRball-14--13) 
  -- (0.294207757259438759090097238491183605141493491,-0.271576391316405008390858989376477173976763223) coordinate (MRball-13--12) 
  -- (0.295123322903217977227949056195353943027235364,-0.270529712661283145792286634845741114441632417) coordinate (MRball-12--11) 
  -- (0.296212728877099994521277799097085907987775491,-0.269284298979181813201161635542805370897977719) coordinate (MRball-11--10) 
  -- (0.297530679746208324579323492773050299938350959,-0.267777611771587492121391143495745269944515863) coordinate (MRball-10--9) 
  -- (0.299157522613365497736841101896203883063632039,-0.265917797878547109099414312796625673834339591) coordinate (MRball-9--8) 
  -- (0.301216266331700661193787321702906017138132167,-0.263564233040238078544563906490042764995865646) coordinate (MRball-8--7) 
  -- (0.303905230425234124941267048526964503264388881,-0.260490197507343535663943184451683859940904755) coordinate (MRball-7--6) 
  -- (0.305583971769715264747093014528218753508176671,-0.258571053035912916324463319985415868353072568) coordinate (MRball-13--11) 
  -- (0.307566095958598793031361815855172738404441961,-0.256305079965498994192801513212643948670368301) coordinate (MRball-6--5) 
  -- (0.309942013504167758926929491619885231675208628,-0.253588920139773620940215038598087916825170695) coordinate (MRball-11--9) 
  -- (0.312842021762280742590348078377887423083271620,-0.250273617409824594072278462702309938466617296) coordinate (MRball-5--4) 
  -- (0.315158946167628318067891709698755894706836526,-0.247624886274565107053343486191879631555371556) coordinate (MRball-14--11) 
  -- (0.316461016651586361061166077335790338259735329,-0.246136346284567169714240282372281374202016367) coordinate (MRball-9--7) 
  -- (0.317875332360723681972268058213302366797076959,-0.244519486431325909209436967856386435997751507) coordinate (MRball-13--10) 
  -- (0.321104233106240949495708206793779769524287895,-0.240828174829680712121781155095334827143215921) coordinate (MRball-4--3) 
  -- (0.325005506497373706672657600863583068588169860,-0.236367641088999059398296436991696777155032625) coordinate (MRball-11--8) 
  -- (0.327277665264113493909476869044564169914852599,-0.233769760902938209935340620746117264224894714) coordinate (MRball-7--5) 
  -- (0.329814018538346439071252283696314392388471297,-0.230869812976842507349876598587420074671929908) coordinate (MRball-10--7) 
  -- (0.331196098835656581570553464033744192045665959,-0.229289606886223787241152398177207517570076433) coordinate (MRball-13--9) 
  -- (0.335887874476123672053504541729161337495115996,-0.223925249650749114702336361152774224996743997) coordinate (MRball-3--2) 
  -- (0.340356387261819756483779905369820196020442302,-0.218800534668312700596715653452027268870284337) coordinate (MRball-14--9) 
  -- (0.341595782285637225424364591317629595300308567,-0.217379134181769143451868376293037015191105452) coordinate (MRball-11--7) 
  -- (0.343786582701475991939320103462710053758970861,-0.214866614188422494962075064664193783599356788) coordinate (MRball-8--5) 
  -- (0.345662405432433962827945025203002661074641288,-0.212715326419959361740273861663386252969010023) coordinate (MRball-13--8) 
  -- (0.348706675059824202196879252800399197956756318,-0.209224005035894521318127551680239518774053791) coordinate (MRball-5--3) 
  -- (0.352065678834931973513594509616969158439497100,-0.205371645987043651216263463943232009089706642) coordinate (MRball-12--7) 
  -- (0.354504861871827477235262324634814254528573914,-0.202574206783901415563007042648465288302042237) coordinate (MRball-7--4) 
  -- (0.357810170976811155921246668489114823013346083,-0.198783428320450642178470371382841568340747824) coordinate (MRball-9--5) 
  -- (0.359945832327618053895782192729806930896955872,-0.196334090360518938488608468761712871398339567) coordinate (MRball-11--6) 
  -- (0.361439363573316640175404212884194379514623506,-0.194621195770247421632909960783796973584797273) coordinate (MRball-13--7) 
  -- (0.369880507348481672999335543754220346515809441,-0.184940253674240836499667771877110173257904720) coordinate (MRball-2--1) 
  -- (0.378329941573791917119144330600727169727636924,-0.174613819187903961747297383354181770643524734) coordinate (MRball-13--6) 
  -- (0.379907849881267037771541782326862923374048923,-0.172685386309666835350700810148574056079113147) coordinate (MRball-11--5) 
  -- (0.382210428271181392865746031776859470922175484,-0.169871301453858396829220458567493098187633549) coordinate (MRball-9--4) 
  -- (0.385885700274704714520333087469587735548871004,-0.165379585832016306222999894629823315235230430) coordinate (MRball-7--3) 
  -- (0.388688874702264792533103386104815926109112596,-0.161953697792610330222126410877006635878796915) coordinate (MRball-12--5) 
  -- (0.392682441523795541565759804543772741470882604,-0.157072976609518216626303921817509096588353042) coordinate (MRball-5--2) 
  -- (0.396442060158443198075602951879374132876029047,-0.152477715445555076182924212261297743413857326) coordinate (MRball-13--5) 
  -- (0.398828599631320465558328188338984946740088968,-0.149560724861745174584373070627119355027533363) coordinate (MRball-8--3) 
  -- (0.401686365340763419844797293547433512474636233,-0.146067769214823061761744470380884913627140448) coordinate (MRball-11--4) 
  -- (0.403337837451371105375160261026241667570808257,-0.144049227661203966205414378937943452703860092) coordinate (MRball-14--5) 
  -- (0.409511194010687331693763792542340360136879266,-0.136503731336895777231254597514113453378959755) coordinate (MRball-3--1) 
  -- (0.416304080947674039363885347151282804299938121,-0.128093563368515089035041645277317785938442499) coordinate (MRball-13--4) 
  -- (0.418386111617740720221037343988761418669329380,-0.125515833485322216066311203196628425600798814) coordinate (MRball-10--3) 
  -- (0.422308509068490776577849425157204608939668478,-0.120659574019568793307956978616344173982762422) coordinate (MRball-7--2) 
  -- (0.425938685635774806048657838534532419468921115,-0.116165096082484038013270319600327023491523940) coordinate (MRball-11--3) 
  -- (0.432443962786746483908554468793527414456939230,-0.108110990696686620977138617198381853614234808) coordinate (MRball-4--1) 
  -- (0.438094308912939744267117392629893229064599238,-0.101098686672216864061642475222283052861061363) coordinate (MRball-13--3) 
  -- (0.440653246360775577004127947310752582731950748,-0.097922943635727906000917321624611685051544611) coordinate (MRball-9--2) 
  -- (0.443056315573411838649534449675720586744348736,-0.094940639051445393996328810644797268588074729) coordinate (MRball-14--3) 
  -- (0.447448543912104553574611160036020202055847434,-0.089489708782420910714922232007204040411169487) coordinate (MRball-5--1) 
  -- (0.453164115712511961677161793453001083996608005,-0.082393475584093083941302144264182015272110546) coordinate (MRball-11--2) 
  -- (0.458039833805152665156227586962105307072095515,-0.076339972300858777526037931160350884512015919) coordinate (MRball-6--1) 
  -- (0.462247873312594786892633715288282008734273309,-0.071115057432706890291174417736658770574503586) coordinate (MRball-13--2) 
  -- (0.465916792398063425082920422744780387750314821,-0.066559541771151917868988631820682912535759260) coordinate (MRball-7--1) 
  -- (0.472004560313869439853110797295278838685975234,-0.059000570039233679981638849661909854835746904) coordinate (MRball-8--1) 
  -- (0.476850598201889314750573269318184423529699783,-0.052983399800209923861174807702020491503299976) coordinate (MRball-9--1) 
  -- (0.480799666338400502874569182094863552689314791,-0.048079966633840050287456918209486355268931479) coordinate (MRball-10--1) 
  -- (0.484079705085712507485269125232608747676772192,-0.044007245916882955225933556839328067970615654) coordinate (MRball-11--1) 
  -- (0.486847452597628157209523606592466317367446831,-0.040570621049802346434126967216038859780620569) coordinate (MRball-12--1) 
  -- (0.489214233247563747024311436815993523580466231,-0.037631864095966442078793187447384117198497402) coordinate (MRball-13--1) 
  -- (0.491261294635998458279092634377480566795749883,-0.035090092473999889877078045312677183342553563) coordinate (MRball-14--1) 
  -- cycle%
}

\title{Counting Markov triples and simple geodesics}
\author{Morten S. Risager}
\address{Department of Mathematical Sciences, University of Copenhagen, Universitetsparken 5, 2100 Copenhagen, Denmark}
\email{risager@math.ku.dk}
\urladdr{https://www.math.ku.dk/~risager/}
\thanks{During the preparation of this work the author was supported by Grant {DFF-3103-0074B}
 from the Independent Research Fund Denmark. }
\date{\today}
\begin{document}
\begin{abstract}We give an asymptotic formula for the number of Markov triples with largest entry below $x$.  
We show that the error term is $o(\log x)$, which improves upon the previous best known results due to McShane and Rivin. We give an explicit error based on classical zero-free regions of the Riemann zeta function, and a sharper estimate under the quasi-Riemann hypothesis. 
  The proof uses the correspondence between Markov triples and simple closed geodesics on the modular punctured torus and a careful analysis of the McShane-Rivin unit ball.
\end{abstract}

\maketitle

\section{Introduction}
A Markov triple is an ordered triple of natural numbers  $0<p\leq q\leq r$ satisfying the equation \begin{equation}p^2+q^2+r^2=3pqr.\end{equation} These are fascinating objects in number theory with a rich history and many connections to various areas of mathematics including Diophantine approximation, hyperbolic geometry, and combinatorics. All Markov triples can be generated from the initial triple $(1,1,1)$ by successively applying Vieta transformations
\begin{equation}(p,q,r)\mapsto (p,r,3pr-q), \text{ and } (p,q,r)\mapsto (q,r,3qr-p). \end{equation} 
See \cite{Bombieri:2007, Aigner:2013, Silverman:2026} for much more about these triples. 

Denote the set of ordered Markov triples by $\Mcal$. We investigate the counting function of Markov triples, i.e. the function 
\begin{equation}M(x)=\#\{(p,q,r)\in \Mcal: r\leq x\}.\end{equation}
Zagier \cite{Zagier:1982} proved that \begin{equation}M(x)=C(\log(3x))^2+O(\log x(\log\log x)^2)\end{equation} for an explicit constant $C\approx 0.1807171047118064780577$.
 This error estimate was improved by McShane and Rivin \cite[Thm. 0.1]{McShaneRivin:1995} \cite{McShaneRivin:1995a} to $O(\log x\log\log x)$. In this paper we prove the first sub-$\log x$ bound on the error: 

\begin{thm} \label{thm:main} As $x\to\infty$, we have 
\begin{equation}M(x)=C(\log(3x))^2+O\left(\frac{\log x(\log\log\log x)^{1/10}}{(\log\log(x))^{3/10}}\right).\end{equation}
If the quasi-Riemann hypothesis holds, then
\begin{equation}M(x)=C(\log(3x))^2+O\left(\frac{\log x}{(\log\log(x))^{1/2}}\right).\end{equation}
\end{thm}

The first bound uses the classical Vinogradov--Korobov estimate \cite[Chap.~V, \S5, Satz~3]{Walfisz:1963}. In fact the prime number theorem alone is enough to give a bound of the form $o(\log x)$.
The quasi-Riemann hypothesis says that there exists a zero-free region for the Riemann zeta function of the form $\Re(s)>1-\delta$ for some fixed $\delta>0$. Such a result was recently announced by OpenAI \cite{OpenAI:2026}, \cite{OpenAI:2026a} claiming $\delta=1/12$, $\delta=1/8$ respectively.

The true size of the error is expected to be much smaller than what is predicted by Theorem \ref{thm:main}. In fact McShane and Rivin conjecture \cite[Conj. 4.3 ]{McShaneRivin:1995a} that the error estimation can be improved to \begin{equation}\label{eq:conjecture_bound}M(x)=C(\log(3x))^2+O_\varepsilon((\log x)^{\tfrac{1}{2}+\varepsilon})\end{equation}  
for every $\varepsilon>0$.

Various comments are in order:
\begin{enumerate}
\item The Frobenius uniqueness conjecture \cite{Frobenius:1913} predicts that the largest coordinate in an ordered Markov triple determines the triple. If true $M(x)$ counts Markov numbers of size at most $x$. The uniqueness conjecture remains open in general. Aigner \cite{Aigner:2013} gives a detailed account of various aspects and known partial results. See also Gaster \cite{Gaster:2022} and Doan, Li and Nguyen \cite{DoanLiNguyen:2026} for more recent progress on the Frobenius conjecture.

\item The above counting problem has a geometric interpretation in terms of simple closed geodesics on the hyperbolic modular punctured torus. McShane and Rivin
\cite{McShaneRivin:1995,McShaneRivin:1995a} showed that, on any complete
finite-area hyperbolic once-punctured torus, the lengths of essential simple closed
geodesics define a norm on real homology. Counting these geodesics amounts to counting primitive lattice points in dilates of
its unit ball with respect to this norm. For the modular torus, Markov
triples with largest entry at most $x$ correspond to isometry orbits of unoriented essential simple closed geodesics of length at most
$L=2\operatorname{arcosh}(3x/2)\sim2\log(3x)$; see \eqref{eq:markov_primitive_count} in Section~\ref{sec:modular-torus} for the exact counting identity, which accounts for the two exceptional orbits. The leading constant in our counting formula therefore comes from the density of primitive lattice points in $\RR^2$, the area of the unit ball $B_X$ with respect to the McShane--Rivin norm, the orbit sizes, and the relation between $L$ and $x$. This is what allowed McShane and Rivin to improve the error term in Zagier's result. McShane and Rivin's result applies to \emph{any} hyperbolic once-punctured torus, and uses only convexity of $B_X$. McShane and Rivin further assert that, for a dense set of hyperbolic once-punctured tori, the error in the quadratic asymptotic for the number of simple closed geodesics of length at most $L$ is not $o(L)$, see \cite[Thm.~4.2]{McShaneRivin:1995a}.   

\item Our improvement of the error term for the modular torus $X$ uses its arithmetic properties to obtain quantitative separation of certain supporting normal directions of $B_X$ from rational directions. Combined with the normal-turn estimates of Doan, Li, and Nguyen \cite{DoanLiNguyen:2026}, this yields a sharper estimate for the number of lattice points in dilates of $B_X$. We then pass to primitive lattice points by M\"obius inversion, using cancellation in sums of the M\"obius function obtained from zero-free regions of the Riemann zeta function. Together, these ingredients give the improved error bounds for $M(x)$.

\item Counting simple geodesics on more general hyperbolic surfaces was pioneered by Mirzakhani \cite{Mirzakhani:2008} who  proved that the number of simple closed geodesics of length at most $L$ on a complete finite-area hyperbolic surface $Y$ of genus $g$ with $n$ cusps where $3g-3+n>0$ is asymptotic to $c_YL^{6g-6+2n}$ for a positive $c_Y$ depending on the surface. Eskin,  Mirzakhani, and Mohammadi \cite{EskinMirzakhaniMohammadi:2022} gave error estimates with a power saving when $Y$ is closed of genus $g\geq 2$.   
\end{enumerate}

\enlargethispage{-\baselineskip}
\subsection{Numerics}
We computed $M(10^k)$ for every integer $1\leq k\leq 30\,000$. Computing these values took about $70$ minutes on a standard laptop. Figure~\ref{fig:normalized-error} shows the normalized error
$\bigl(M(x)-C(\log(3x))^2\bigr)/\sqrt{\log x}$ at $x=10^k$. Its sampled values lie between approximately $-1.113$ and $1.150$. This seems consistent with \eqref{eq:conjecture_bound}.

\begin{figure}[htbp]
\centering
\includegraphics[width=\linewidth]{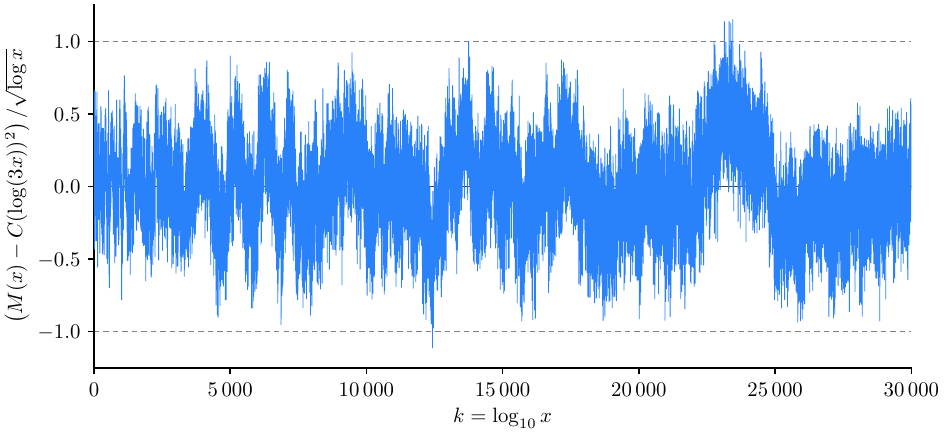}
\caption{Normalized error term}
\label{fig:normalized-error}
\end{figure}

The largest numerical verification of the uniqueness conjecture we could find in the literature appears in \cite[App.~F]{Perling:2022}: it covers all Markov numbers less than $10^{15000}$ (see \cite{RosenPatterson:1971a,Borosh:1975,Baragar:1996} for earlier experiments). We verified the uniqueness conjecture numerically up to $10^{30000}$. The computation took about $90$ minutes on a standard laptop. In fact, we have verified the slightly stronger statement that there are no two different ordered Markov triples $(p,q,r)$, $(p',q',r')$ with $r,r'\leq 10^{30000}$ satisfying $r=r'\pmod{2^{128}}$. We found $M(10^{30000})=862\,356\,773$, while the main term in the asymptotic formula predicts about
$862\,356\,738$ triples.

\section{The modular punctured torus and simple geodesics}
\label{sec:modular-torus}
We now recall the connection between Markov triples and simple closed geodesics on the modular punctured torus. This connection goes back at least to Cohn \cite{Cohn:1955} and has been used in many papers including \cite{Zagier:1982, Series:1985a, McShaneRivin:1995,McShaneRivin:1995a, Gaster:2022, Springborn:2018} and others.
Consider the modular commutator group 
$\Gamma=[\mathrm{PSL}_2(\ZZ),\mathrm{PSL}_2(\ZZ)]$ and the associated quotient
\begin{equation}
\qquad X=\Gamma\backslash\mathbb H.
\end{equation}
The group $\Gamma$ has index six in $\mathrm{PSL}_2(\ZZ)$, and
$X$ is a complete hyperbolic punctured torus. The puncture is a cusp at infinity
\cite{Newman:1962},\cite[Sec.~12]{Springborn:2018}.
The images of the matrices
\begin{equation}\label{eq:modular_generators}
A=\begin{pmatrix}1&-1\\-1&2\end{pmatrix},
\qquad
B=\begin{pmatrix}1&1\\1&2\end{pmatrix}
\end{equation}
generate $\Gamma$. Note that 
\begin{equation}
\trace A=\trace B=\trace(AB)=3,
\qquad
[A,B]=\begin{pmatrix}-1&0\\6&-1\end{pmatrix}.
\end{equation}
The commutator represents a loop around the puncture and has
trace $-2$.

\begin{figure}[htbp]
\centering
\begin{tikzpicture}[x=2.2cm,y=2.2cm,
  seam/.style={densely dashed,thin,black!60},
  paired/.style={semithick,postaction={decorate}},
  paired A/.style={paired,black,
    decoration={markings,
      mark=at position .55 with {\arrow{Stealth}}}},
  paired B/.style={paired,black,
    decoration={markings,
      mark=at position .53 with {\arrow{Stealth}},
      mark=at position .57 with {\arrow{Stealth}}}}]
  \fill[FigureBlue!18]
    (-1,2.05) -- (-1,0) arc[start angle=180,end angle=0,radius=.5]
    arc[start angle=180,end angle=0,radius=.5] -- (1,2.05) -- cycle;
  \draw[seam] (-.5,.5) -- (-.5,2.1);
  \draw[seam] (.5,.5) -- (.5,2.1);
  \draw[seam] (-1,0) arc[start angle=180,end angle=0,radius=1];
  \draw[seam] (0,0) arc[start angle=180,end angle=90,radius=1];
  \draw[seam] (-1,1) arc[start angle=90,end angle=0,radius=1];
  \draw[paired A] (1,0) -- (1,1.75) node[pos=.76,right] {$s_A$};
  \draw[paired A] (0,0) arc[start angle=0,end angle=180,radius=.5]
    node[pos=.5,below=5pt] {$s_A'$};
  \draw[paired B] (-1,0) -- (-1,1.75) node[pos=.76,left] {$s_B$};
  \draw[paired B] (0,0) arc[start angle=180,end angle=0,radius=.5]
    node[pos=.5,below=5pt] {$s_B'$};
  \draw[densely dashed,thin,black] (-1,1.75) -- (-1,2.1);
  \draw[densely dashed,thin,black] (1,1.75) -- (1,2.1);
  \draw[thin] (-1.35,0) -- (1.35,0) node[right] {$\RR$};
  \foreach \x/\lab in {-1/-1,0/0,1/1}
    \node[below=3pt] at (\x,0) {$\lab$};
\end{tikzpicture}
\par\smallskip
{\small $A\cdot s_A=s_A',\qquad B\cdot s_B=s_B'.$\par}
\caption{Fundamental domain for $\Gamma$.}
\label{fig:modular-fundamental-domain}
\end{figure}
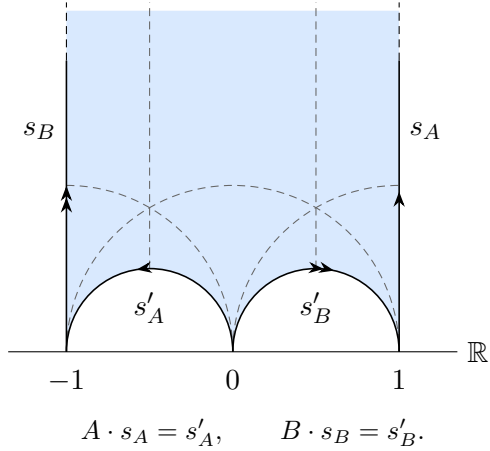

Choose the identification 
\begin{equation}\label{eq:homology_identification}
\begin{aligned}
H_1(X,\ZZ)
&\cong \ZZ^2 \\
\gamma_u
&\mapsfrom u.
\end{aligned}
\end{equation}
for which the
classes of $A$ and $B$ form the standard basis. The free homotopy
classes of oriented, essential, simple, closed curves  are parametrized by primitive vectors
\begin{equation}
\ZZ^2_{\mathrm{prim}}
=\{(p,q)\in\ZZ^2:\gcd(p,q)=1\}.
\end{equation} Here essential means that the curve is not homotopic to a point or to the puncture/cusp.

Each class has a unique simple closed geodesic representative
$\gamma_v$ \cite[Sec.~2]{McShaneRivin:1995}.
The geometric intersection number of the curves represented by
$u$ and $v$ is $|\det(u,v)|$
\cite[Sec.~1.2.3]{FarbMargalit:2012}. We call two primitive
vectors modulo sign \emph{Farey neighbours} if
$|\det(u,v)|=1$; they correspond to curves intersecting once.
A \emph{Farey triple} consists of three such classes that are
pairwise Farey neighbours. Equivalently, after choosing signs,
it has representatives $u,v,u+v$ with $|\det(u,v)|=1$.

The connection with the Markov equation comes from the Fricke
trace identity. For $P,Q\in\mathrm{SL}_2(\RR)$, set
\begin{equation}
a=\trace P,\qquad b=\trace Q,\qquad c=\trace(PQ).
\end{equation}
Then \cite[Eq.~(2.2.8)]{Goldman:2009} gives
\begin{equation}
a^2+b^2+c^2-abc=\trace([P,Q])+2.
\end{equation}
For a Farey triple represented by $u,v,u+v$, base $\gamma_u$ and
$\gamma_v$ at their unique intersection point. These loops form
a free generating pair, with product representing $\gamma_{u+v}$. Taking their
holonomy matrices as $P,Q$, the traces $a,b,c$ are those of the
three curves. The commutator represents a loop going around the cusp, and has trace $-2$,
so the right-hand side is zero. On $X$, the lift of the holonomy
determined by $A$ and $B$ has positive traces on simple closed
curves, and these traces are three times Markov numbers
\cite{Cohn:1955}\cite[Sec.~2]{Gaster:2022}.
Consequently, the normalized traces $(a/3,b/3,c/3)$ of a
Farey triple satisfy the Markov equation.

This correspondence respects the recursive constructions.
For $w\in\ZZ^2_{\mathrm{prim}}$, write $t_w=\trace(\gamma_w)$
for the positive holonomy trace of $\gamma_w$. If $|\det(u,v)|=1$, the two Farey
triangles adjacent to the edge joining $u$ and $v$ have third
vertices $u+v$ and $u-v$, considered modulo sign. The Fricke identity
\cite[Thm.~2.3]{DoanLiNguyen:2026} and the trace equation above give
\begin{equation}\label{eq:Fricke_identity}
t_{u+v}+t_{u-v}=t_ut_v,\quad\quad  
t_{u+v}t_{u-v}=t_u^2+t_v^2.
\end{equation}

Passing from the triangle with vertices $u,v,u+v$ to the adjacent
triangle with vertices $u,v,u-v$ keeps the normalized traces of $u$ and $v$
fixed. If $(x,y,z)=(t_u/3,t_v/3,t_{u+v}/3)$, then
the first equality in \eqref{eq:Fricke_identity} gives $t_{u-v}/3=3xy-z$. Hence this move changes
the normalized trace coordinates by the Vieta transform
\begin{equation}
(x,y,z)\longmapsto(x,y,3xy-z).
\end{equation}
The initial triangle represented by $A$, $B$, and $AB$ has
normalized traces $(1,1,1)$. The resulting recursion recovers
the Markov tree; compare
\cite[Prop.~12.1]{Springborn:2018}.

Unoriented simple closed geodesics, modulo the isometries of
$X$, correspond to Markov triples with their entries in
nondecreasing order \cite[Secs.~12--15]{Springborn:2018}.
The triple $(x,y,z)$ corresponds to the orbit whose geodesics
have trace $3z$. The trace--length relation
\begin{equation}\label{eq:trace_length}2\cosh(\ell_X(\gamma)/2)=|\trace\gamma|
\end{equation}
\cite[Sec.~2]{Gaster:2022} therefore gives
\begin{equation}\label{eq:markov_length}
L_z=2\operatorname{arcosh}(3z/2).
\end{equation}

The full isometry group of $X$ has order twelve
\cite[Sec.~12]{Springborn:2018}.
Its central elliptic
involution fixes every unoriented simple closed geodesic, so
the action on primitive vectors modulo sign factors through
a dihedral group of order six. This action is free except on
the two orbits corresponding to $(1,1,1)$ and $(1,1,2)$,
which contain three geodesics each. Every other orbit contains
six geodesics; see the discussion of the exceptional symmetry
stabilizers in \cite[Sec.~2]{Fisac:2025}.

Let $S_X(L)$ denote the number of unoriented simple closed
geodesics on $X$ of length at most $L$. For $T\ge2$, the
two exceptional triples are both included in $M(T)$. Each contributes
three rather than six geodesics, so the preceding correspondence gives
the exact identity
\begin{equation}\label{eq:markov_geodesic_count}
S_X(L_T)=6\bigl(M(T)-2\bigr)+2\cdot 3,
\qquad L_T=2\operatorname{arcosh}(3T/2)=2\log(3T)+O(T^{-2}).
\end{equation}

The next section introduces the McShane--Rivin norm. Its defining feature is that it
agrees with $\ell_X(\gamma_v)$ on primitive integral classes
\cite[Sec.~5]{McShaneRivin:1995}. Writing $B_X$ for its unit
ball, \eqref{eq:markov_geodesic_count} becomes
\begin{equation}\label{eq:markov_primitive_count}
M(T)=\frac1{12}\#\bigl(L_TB_X\cap\ZZ^2_{\mathrm{prim}}\bigr)+1.
\end{equation}
This reduces the counting of Markov triples to the counting
of primitive lattice points in dilates of a fixed convex body.

\section{McShane-Rivin norms}\label{sec:McShane-Rivin-norms}
In the previous section we described the correspondence between Markov triples and simple closed geodesics on the modular punctured torus. In this section we describe the McShane-Rivin norm, which is a key tool in counting these geodesics. We start by recalling the definition of the McShane-Rivin norm and its basic properties. This is not restricted to the modular punctured torus, but applies to $X$ any complete finite-area hyperbolic once-punctured torus.

Choose an identification $H_1(X,\ZZ)\cong\ZZ^2$. For a
primitive vector $v\in \ZZ^2_{\mathrm{prim}}$, let $\gamma_v$ be its simple closed geodesic
representative. McShane and Rivin
\cite{McShaneRivin:1995,McShaneRivin:1995a} showed that there
is a norm $\MRnorm{\cdot}:\RR^2\to\RR$ satisfying that for all such $v$ we have
\begin{equation}
\MRnorm{v}=\ell_X(\gamma_v).
\end{equation}
On integral multiples, the norm is determined by homogeneity:
\begin{equation}
\MRnorm{dv}=|d|\ell_X(\gamma_v)\qquad(d\in\ZZ).
\end{equation}
These values extend uniquely to a norm on real homology.
By equivalence of norms, there exist constants $c,C>0$ such that for every $x\in \RR^2$ 
\begin{equation}\label{eq:equivalence_of_norms}c \abs{x}\leq \MRnorm{x}\leq C\abs{x}.
\end{equation}
where $\abs{\cdot}$ is the Euclidean norm.
Let
\begin{equation}
B_X=\{x\in\RR^2:\MRnorm{x}\leq1\}
\end{equation}
be the unit ball with respect to the $\MRnorm{\cdot}$ norm.
We call $B_X$ the McShane--Rivin ball. It is a centrally
symmetric, strictly convex compact set containing the origin in its interior. 
\cite{McShaneRivin:1995,McShaneRivin:1995a, DoanLiNguyen:2026}. Contrary to what Figure \ref{fig:mcshane-rivin-ball} might suggest this ball has corners at all rational directions. 

Define the full and primitive lattice-point counts by
\begin{equation}\label{eq:lattice_point_counting_functions}
N_X(R)=\#(RB_X\cap\ZZ^2),
\qquad
P_X(R)=\#(RB_X\cap\ZZ^2_{\mathrm{prim}}).
\end{equation}
The number of oriented simple closed geodesics on $X$ of
length at most $R$ is $P_X(R)$.

An overview of the rich history of counting lattice points in convex domains can be found in \cite{IvicKratzelKuhleitnerNowak:2006}. From the fact that $B_X$ is a convex domain it follows from results of Jarn\'{\i}k and Steinhaus (See \cite[Sec 3.1]{IvicKratzelKuhleitnerNowak:2006}) that  
\begin{equation}\label{eq:convex_bound_JS}
  N_X(R)=\vol(B_X)R^2+O(R)\end{equation}
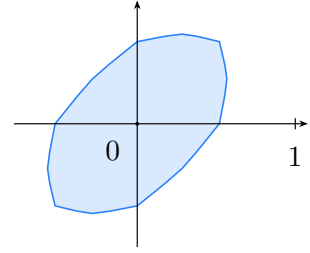
\begin{wrapfigure}{r}{0.42\textwidth}
\centering
\scalebox{1.1}{%
\begin{tikzpicture}[x=1.9cm,y=1.9cm]
  \begin{scope}
    \clip (-0.8,-0.8) rectangle (1.1,0.8);
    \filldraw[fill=FigureBlue!18,draw=FigureBlue,line width=.55pt]
    \MRBallBoundary;
    \draw[thin,-{Stealth[length=3pt]}] (-0.78,0) -- (1.08,0);
    \draw[thin,-{Stealth[length=3pt]}] (0,-0.78) -- (0,0.78);
    \draw[thin] (1,-0.035) -- (1,0.035);
    \node[below=4pt,font=\small] at (1,0) {$1$};
    \node[below left=2pt,font=\small] at (0,0) {$0$};
    \fill[black] (0,0) circle[radius=.65pt];
  \end{scope}
\end{tikzpicture}%
}
\caption{$B_X$ for the modular torus}
\label{fig:mcshane-rivin-ball}
\end{wrapfigure}
where the implied constant is bounded by the length of the boundary of $B_X$. Extra smoothness of a convex region usually results in better error estimates with the Gauss circle problem as the ultimate goal (See e.g. \cite{Huxley:1996,IvicKratzelKuhleitnerNowak:2006})  The error estimation \eqref{eq:convex_bound_JS} is the basis for the error estimation by McShane and Rivin \cite[Thm. 0.1]{McShaneRivin:1995}. For general convex domains \eqref{eq:convex_bound_JS} is best possible and is attained by the square. 

In this paper we  show that for the McShane-Rivin ball $B_X$ for the once-punctured modular torus we can get a better error term. 
\WFclear

Understanding the geometry of $B_X$ and its boundary $\partial B_X$ is crucial in doing so. We shall rely crucially on the recent work of Doan, Li and Nguyen \cite{DoanLiNguyen:2026}.

\subsection{Support directions and normal turn} \label{sec:support_and_normal_turn}
We now introduce some geometric notation and terminology which will be used in the rest of the paper. Here we largely follow \cite{DoanLiNguyen:2026}.
\subsubsection{Functionals:} Consider the vector space of linear functionals $(\RR^2)^*=\hbox{Hom}(\RR^2,\RR)$. This can be identified with $\RR^2$ in the usual way via the linear isomorphism
\begin{equation}
\begin{aligned}
\mathbb{R}^{2}
&\longrightarrow \operatorname{Hom}(\mathbb{R}^{2},\mathbb{R}) \\
\lambda^{\#}
&\longmapsto \lambda=\langle \lambda^{\#},\cdot \rangle.
\end{aligned}
\end{equation}
Here $\langle\cdot,\cdot\rangle$ is the standard Euclidean inner product on $\RR^2$. The vector $\lambda^{\#}$ is called the metric dual of the covector $\lambda$. We let $\abs{\lambda}=\abs{\lambda^{\#}}$. Furthermore, define for covectors $\lambda_1, \lambda_2\in (\RR^2)^*$ the linear map $T_{\lambda_1,\lambda_2}:\RR^2\to\RR^2$ by $T_{\lambda_1,\lambda_2}(x)=(\lambda_i(x))_{i=1,2}$ and define
\begin{equation}
    \det(\lambda_1,\lambda_2):=\det T_{\lambda_1,\lambda_2}=\det(\lambda_1^{\#},\lambda_2^{\#}).
\end{equation}
If $\lambda_1,\lambda_2\ne0$ and $\theta\in[0,\pi]$ is the angle between their metric duals, then
$\abs{\det(\lambda_1,\lambda_2)}=\abs{\lambda_1}\abs{\lambda_2}\sin\theta$.
For linearly independent $u,v\in\RR^2$, write $[u\ v]$ for the matrix
with columns $u$ and $v$, and $\det(u,v)=\det[u\ v]$.
Identifying $T_{\lambda_1,\lambda_2}$ with its matrix in the standard basis,
we have
\begin{equation}
T_{\lambda_1,\lambda_2}[u\ v]
=\begin{pmatrix}
\lambda_1(u)&\lambda_1(v)\\
\lambda_2(u)&\lambda_2(v)
\end{pmatrix}.
\end{equation}
Taking determinants gives
\begin{equation}\label{eq:determinant_calculation}
\det(\lambda_1,\lambda_2)
=\frac{\lambda_1(u)\lambda_2(v)-\lambda_2(u)\lambda_1(v)}{\det(u,v)}.
\end{equation}

\subsubsection{Supporting lines:} 
For $\lambda\in (\RR^2)^*$ the line $\lambda=1$ is called a supporting line of $B_X$ at $x\in\partial B_X$ if $\lambda(x)=1$ and $\lambda(y)\leq 1$ for all $y\in B_X$. We denote the set of supporting functionals at $x$ by
\begin{equation}\partial _{B_X}(x):=\{\lambda\in (\RR^2)^* : \lambda(x)=1,\lambda(y)\leq 1 \text{ for all }y\in B_X\}.\end{equation}
As in \cite{DoanLiNguyen:2026}, for $0\neq a\in\RR^2$ we define the set of supporting functionals in the direction $a$ by $\partial B_X(a)=\partial _{B_X}(a/\MRnorm{a})$. For $\lambda\in \partial _{B_X}(x)$ with metric dual $\lambda^{\#}$ we have that \begin{equation}n(\lambda):=\frac{\lambda^{\#}}{\abs{\lambda^{\#}}}\end{equation} is the unit outward normal vector of the supporting line $\lambda=1$. For any $y\in \RR^2$ we have 
\begin{equation}
\abs{\lambda(y)}\leq \MRnorm{y}.\end{equation} 

\noindent
\begin{minipage}[t]{.49\linewidth}
\vspace{0pt}
We next define the supporting functional selected by a second direction. For $a,b\in \RR^2$ linearly independent define $\lambda_{a\vert b}\in \partial B_X(a)$ to be the unique supporting functional at $a$ such that \begin{equation}\lambda_{a\vert b}(b)=\max_{\lambda\in \partial B_X(a)}\lambda(b).\end{equation} This is unique: if two functionals agree on both $a$ and $b$, they are equal because $a$ and $b$ form a basis of $\RR^2$.
\end{minipage}\hfill
\begin{minipage}[t]{.48\linewidth}
\vspace{0pt}
\centering
\makeatletter\def\@captype{figure}\makeatother
\begin{tikzpicture}[x=2.55cm,y=2.55cm]
  \begin{scope}
    \clip (-.7,-.1) rectangle (.7,.75);
    \filldraw[fill=FigureBlue!18,draw=FigureBlue,line width=.5pt]
      \MRBallBoundary;
  \end{scope}
  \draw[thin,black!55,-{Stealth[length=3pt]}] (-1.05,0) -- (1.08,0)
    node[below right,font=\scriptsize] {$x$};
  \draw[thin,black!55,-{Stealth[length=3pt]}] (0,-.1) -- (0,1.02)
    node[above left,font=\scriptsize] {$y$};
  \coordinate (P) at (MRball-1-1);
  \coordinate (Q) at (MRball-1-2);
  \draw[densely dashed,black!55] (0,0) -- (P);
  \draw[densely dashed,black!55] (0,0) -- (Q);
  \draw[FigureBlue,line width=1pt]
    (P)
    foreach \p/\q in {
      13/14,12/13,11/12,10/11,9/10,8/9,7/8,6/7,
      11/13,5/6,9/11,4/5,11/14,7/9,10/13,3/4,
      8/11,5/7,7/10,9/13,2/3,9/14,7/11,5/8,
      8/13,3/5,7/12,4/7,5/9,6/11,7/13,1/2} { -- (MRball-\p-\q) };
  \draw[black,line width=.8pt]
    (-1.05,0.898825676088198) -- (1.08,0.384071793482650)
    node[pos=.24,above=3pt,sloped,font=\scriptsize] {$\lambda_{u\vert v}=1$};
  \draw[thin,black!65] (P) -- (.69,.59);
  \node[anchor=west,font=\scriptsize] at (.69,.59) {$\dfrac{u}{\MRnorm{u}}$};
  \draw[thin,black!65] (Q) -- (.10,.49) -- (-.16,.49);
  \node[anchor=east,font=\scriptsize] at (-.16,.49) {$\dfrac{v}{\MRnorm{v}}$};
  \fill[red!75!black] (P) circle[radius=1.4pt];
  \fill[FigureBlue] (Q) circle[radius=1.3pt];
  \fill[black] (0,0) circle[radius=.7pt];
  \node[below left=2pt,font=\scriptsize] at (0,0) {$0$};
\end{tikzpicture}
\setlength{\captionindent}{0pt}
{\small\caption{$\lambda_{u\vert v}=1$,  $u=(1,1)$,
$v=(1,2)$.}}
\label{fig:supporting-line-u-v}
\end{minipage}

By \cite[Lem 2.6]{DoanLiNguyen:2026} we have 
\begin{equation}\frac{\MRnorm{a+tb}-\MRnorm{a}}{t}\to \lambda_{a\vert b}(b)\text{ for }t\to 0^+.
\end{equation}

Let $\muturn$ be the positive Borel measure on $\RR^2$ that measures
the angular change in supporting normals of $B_X$. It is supported on
$\partial B_X$ and normalized by $\muturn(\partial B_X)=2\pi$. See \cite{DoanLiNguyen:2026} and \cite[p. 221 ff.]{Schneider:1998}
for additional background on $\muturn$. The measure is denoted by $\hbox{Turn}(\cdot)$ in \cite{DoanLiNguyen:2026} and by $C_0(B_X,\cdot)$ in \cite{Schneider:1998}.

For a primitive vector $p\in\ZZ^2$, we call $\infnorm{p}$ the height of both $p$ and the corresponding boundary point $p/\MRnorm{p}$. Consider the corners of $B_X$ coming from primitive lattice points of height at most $H$, i.e. \begin{equation}\label{eq:corners_definition}
C_H:=\{\tfrac{p}{\MRnorm{p}}\in \partial B_X:p\in \ZZ^2_{\mathrm{prim}}, \infnorm{p}\leq H\}.    
\end{equation}
Then 
\begin{equation}\partial B_X\backslash C_H\end{equation}
consists of $\# C_H$ connected components $\Gamma_{H,j}$ of Euclidean lengths $s_{H,j}$. The two endpoints of each component lie in $C_H$ and have heights as defined above. Here we have ordered the arcs cyclically, i.e. such that $\overline{\Gamma_{H,j}}$ and $\overline{\Gamma_{H,j+1}}$ share a common endpoint. We also  denote the set of angular directions of supporting normals of $B_X$ in  $\Gamma_{H,j}$ by \begin{equation}J_{H,j}\subseteq S^1
\end{equation}
Therefore $J_{H,j}$ is an open connected arc of length $\muturn(\Gamma_{H,j}).$  We denote the connected arc length of $J_{H,j}$ by $\abs{J_{H,j}}$, so that $\muturn{\Gamma_{H,j}}=\abs{J_{H,j}}$.

The following result about once-punctured hyperbolic tori is due to Doan, Li, and Nguyen.
\begin{thm}[{\cite[Thm.~1.1 and Prop.~4.1]{DoanLiNguyen:2026}}] \label{thm:DoanLiNguyen} There exist positive constants $\kappa$, $\Kappa$, $m$, $M$ depending only on $X$ such that for $H\geq 1$ we have 
  \begin{enumerate}
    \renewcommand{\theenumi}{\roman{enumi}}
    \renewcommand{\labelenumi}{\theenumi)}
    \item \label{thm:DoanLiNguyen_uno}$\muturn(\partial B_X\backslash  C_H)\leq M e^{-\kappa H}$
    \item \label{thm:DoanLiNguyen_duo}$ \muturn(\{w\})\geq m e^{-\Kappa H} \text{ for } w\in C_H .$
  \end{enumerate}
  \end{thm}
This says that after removing the corners coming from primitive lattice points of height at most $H$, exponentially little turn remains.  On the other hand the turn at these corners  can all be controlled from below.  In fact a more precise statement is proved in \cite[Prop 4.1]{DoanLiNguyen:2026}, namely that $\muturn(\{p/\MRnorm{p}\})$ for $p\in \ZZ^2_{\mathrm{prim}}$ is of order $\infnorm{p}e^{-\MRnorm{p}}$ but we will not need this stronger result here.

We next show an elementary result about the number of arcs and their lengths:
\begin{lemma}\label{lem:arc_length} The number of arcs equals $\# C_H=O(H^2)$. If $A$ is the smaller of the two endpoint heights of $\Gamma_{H,j}$, then $s_{H,j}\ll (HA)^{-1}$. 
\end{lemma}
\begin{proof}The first statement is clear since it counts the number of primitive lattice points in the square $[-H,H]^2$, of side length $2H$. The second follows from the following argument: 
The four lines $x=0$, $y=0$, $y=x$, and $y=-x$ divide $\RR^2$ into
eight coordinate sectors. Every boundary ray of these sectors is
represented by a primitive vector of height $1$, so for $H\geq 1$ the arc $\Gamma_{H,j}$  is
contained in precisely one such closed coordinate sector.

Choose a signed coordinate permutation $T$ carrying this sector to
\begin{equation}
 \mathcal C=\{(p,q):0\le q\le p\}.
\end{equation}
Applying $T$ simultaneously to the two endpoint representatives,
write them as $(p,q)$ and $(p',q')$, where
$0\le q\le p$ and $0\le q'\le p'$. The map $T$ preserves $\ZZ^2$,
primitivity, absolute determinants, height, and Euclidean distances.

In the transformed coordinates put
\begin{equation}
\lVert z\rVert_{X,T} =\MRnorm{T^{-1}z}.
\end{equation}
This is a norm with unit ball $T(B_X)$. Since $T$ is a Euclidean isometry,
$\abs{T^{-1}z}=\abs{z}$, and \eqref{eq:equivalence_of_norms} gives
\begin{equation}
 c\abs{z}\le \lVert z\rVert_{X,T}\le C\abs{z}
\end{equation}
with the same constants as for $\MRnorm{\cdot}$. Therefore the norm and metric
estimates below are uniform over the eight sectors.

In $\mathcal C$ the slopes are the Farey fractions $q/p$ of order
$H$. Consecutive fractions $q/p,q'/p'$ satisfy
\begin{equation}
 1\le p,p'\le H,\quad p+p'>H,\quad |pq'-qp'|=1,
 \quad \left|\frac qp-\frac{q'}{p'}\right|=\frac1{pp'}.
\end{equation} by \cite[III, Sec 3.1]{HardyWright:2008}.
The map $t\mapsto(1,t)/\lVert (1,t)\rVert_{X,T}$ traces out the boundary of $T(B_X)$ in $\mathcal C$. It is Lipschitz on $[0,1]$,
uniformly in $T$: the norms $\lVert\cdot \rVert_{X,T}$ have a common Lipschitz bound and
$\lVert(1,t) \rVert_{X,T}\ge c$. Its arc length over an interval is therefore at
most a fixed constant times the interval length, so the arc length of $T\Gamma_{H,j}$ is at most a bounded constant times $\left|\frac qp-\frac{q'}{p'}\right|$.   Since $T$ is a
Euclidean isometry, it preserves the length of the original boundary
arc. Thus $s_{H,j}\ll1/(pp')$.
As $\max(p,p')>H/2$ and $\min(p,p')=A$, this proves the claim. 
\end{proof} 
For a general hyperbolic once-punctured torus, we use the same notation $t_u$
for the absolute holonomy trace of $\gamma_u$. By \eqref{eq:trace_length},
for $u\in\ZZ^2_{\mathrm{prim}}$ we have
\begin{equation}\label{eq:t_u_def}
    t_u=\abs{\trace(\gamma_u)}=2\cosh(\tfrac{\MRnorm{u}}{2}).
\end{equation}
For $u,v\in \ZZ^2_{\mathrm{prim}}$ with $\abs{\det(u,v)}=1$ we set \begin{equation}t_+=t_{u+v}, \quad t_-=t_{u-v},\end{equation} so that the Fricke identity \eqref{eq:Fricke_identity} gives $t_+ + t_- = t_u t_v.$ 
Set  \begin{equation}\label{eq:r_def}r=e^{\MRnorm{u}/2}>1.\end{equation}  Set \begin{equation}A_\pm=\frac{t_\pm -r^{-1}t_v}{r-r^{-1}}.\end{equation} As in the proof of \cite[Prop 4.1]{DoanLiNguyen:2026}  we observe that  $A_\pm>0$ and \begin{equation}
\lambda_{u\vert v}(v)=2\log A_+\quad \lambda_{u\vert -v}( v)=-2\log A_-.\end{equation}
Note that \begin{equation}\lambda_{u\vert v}(u)=\MRnorm{u}\lambda_{u\vert v}(\tfrac{u}{\MRnorm u})=\MRnorm{u}=2\log r\end{equation}
In the next proposition we specialize to the modular punctured torus.
Proposition~\ref{prop:elementary-prop} (used in Lemma~\ref{lem:support_separation})
is the only place where the arithmetic of $X$ enters.

\begin{prop}\label{prop:elementary-prop} Let $X$ be the modular punctured torus, and set $\Delta=t_u^2-4$. 
 Let $K=\QQ(\sqrt{\Delta})$ be the corresponding real quadratic field, $\mathcal O_K$ its ring of integers, and $\sigma$ its non-trivial field automorphism. Then  
\begin{enumerate}
\renewcommand{\theenumi}{\roman{enumi}}
\renewcommand{\labelenumi}{(\theenumi)}
\item $r$ is a unit in $\mathcal O_K$,
\item \label{prop:nt_ii} $A_\pm \in K$ and $\norm_{K/\QQ}(A_\pm)=\frac{t_u^2}{t_u^2-4}>1$, 
\item \label{prop:nt_iii} $\trace_{K/\QQ}(A_\pm)=A_\pm+\sigma(A_\pm)=t_v$,
\item $\Delta A_\pm$ are algebraic integers in $K$. 

\end{enumerate}
\end{prop}
\begin{proof}Note that $t_u^2-4>0$ is not a square since $t_u\in\ZZ$ and $t_u\ge3$, so the field is indeed real quadratic. We have $r+r^{-1}=t_u$, so $r$ is the larger root of $x^2-t_ux+1$, which gives $r=(t_u+\sqrt{\Delta})/2$ and $\sigma(r)=r^{-1}$. It follows that $r$ is a unit in $\mathcal{O}_K$. Also we have
    $A_-=\tfrac{rt_v-t_+}{r-r^{-1}}$ by the Fricke identity,
    and $\sigma(A_\pm)=A_\mp$. It follows that \begin{equation}
     \trace_{K/\QQ}(A_\pm)= A_++A_-=\frac{t_+-r^{-1}t_v+(rt_v-t_+)}{r-r^{-1}}=t_v.
      \end{equation}
The second relation in \eqref{eq:Fricke_identity} gives
$t_+t_-=t_u^2+t_v^2$. Hence
\begin{equation}
\begin{aligned}
\norm_{K/\QQ}(A_\pm)=A_+A_-
&=\frac{(t_+-r^{-1}t_v)(t_--r^{-1}t_v)}{(r-r^{-1})^2}\\
&=\frac{t_+t_--t_v^2}{(r-r^{-1})^2}
=\frac{t_u^2}{\Delta}.
\end{aligned}
\end{equation}
                    
To see that $\Delta A_\pm$ are algebraic integers, it suffices to note that they are roots of the monic polynomial with integer coefficients \begin{equation}x^2-\Delta t_v x+\Delta t_u^2.\end{equation} 
\end{proof}

A covector $d\in(\RR^2)^*$ with integer dual vector $d^\#\in\ZZ^2$ is called \emph{integral}. We denote the set of integral covectors by $(\ZZ^2)^*$. The next lemma is key in proving our main theorem: it gives a quantitative separation between the supporting functional selected by a Farey pair and any nonzero integral functional.

\begin{lemma}\label{lem:support_separation}
Let $X$ be the modular punctured torus. There exists a constant $M_1>0$ such that, for every $u,v\in\ZZ^2_{\mathrm{prim}}$ with $\infnorm{u}=A\geq1$, $\infnorm{v}\leq2A$, and $\abs{\det(u,v)}=1$, and every nonzero integral $d\in (\ZZ^2)^*$, one has
\begin{equation}\abs{\det(\lambda_{u\vert v},d)}\geq e^{-M_1\abs{d}A^2}.
\end{equation}
\end{lemma}
\begin{proof}
    Since $\abs{\det(u,v)}=1$, equation \eqref{eq:determinant_calculation} gives
    \begin{align}\abs{\det(\lambda_{u\vert v},d)}&= \abs{\lambda_{u\vert v}(u)d(v)-\lambda_{u\vert v}(v)d(u)}\\&=\abs{2\log(r) d(v)-2\log (A_+) d(u)}=2\abs{\log (r^{d(v)}A_+^{-d(u)})}.\label{eq:determinant_calculation_2}
    \end{align}

If $\abs{\det(\lambda_{u\vert v},d)}=0$, then $r^{d(v)}=A_+^{d(u)}$. Since $d$ is integral, both exponents are integers. Taking field norms and using Proposition~\ref{prop:elementary-prop}\eqref{prop:nt_ii} gives \begin{equation}1=\left(\frac{t_u^2}{t_u^2-4}\right)^{d(u)}.\end{equation} Hence $d(u)=0$, and then $r^{d(v)}=A_+^{d(u)}=1$. Since $r>1$, we also have $d(v)=0$. As $u,v$ form a basis, this gives $d=0$, a contradiction. In particular we conclude that
\begin{equation}\label{eq:not_a_one}
    r^{d(v)}A_+^{-d(u)}\neq 1.
\end{equation}

By \eqref{eq:t_u_def} and norm equivalence on $\RR^2$, we have
\begin{equation}\label{eq:bound1}
    \log t_u, \log t_v\ll A, \quad \abs{d(u)}, \abs{d(v)}\ll \abs{d}A.
\end{equation}
The implied constant depends only on $X$. Proposition~\ref{prop:elementary-prop}\eqref{prop:nt_ii} and \eqref{prop:nt_iii}, together with $A_\pm>0$, give
\begin{equation}\label{eq:bound2}
\tfrac{1}{t_v}<\tfrac{1}{A_\mp}<\tfrac{A_+A_-}{A_\mp}=A_\pm<t_v. 
\end{equation} Also note that $r<t_u$.

Let \begin{equation}n_+=\max(0,d(u)),\quad  n_-=\max(0,-d(u)), \quad \text{ and }D=t_u^{2n_+}\Delta^{n_-}.\end{equation} Let $\beta=r^{d(v)}A_+^{-d(u)}$. Then $\beta>0$, and \begin{equation}
D\beta=r^{d(v)}(\Delta A_-)^{n_+}(\Delta A_+)^{n_-}\end{equation} so by Proposition \ref{prop:elementary-prop} we have $D\beta\in \Ocal_K$, and as a consequence $D(\beta-1)\in \Ocal_K$. Moreover, $\sigma(\beta)=r^{-d(v)}A_-^{-d(u)}$. Using $r<t_u$ together with
\eqref{eq:bound1} and \eqref{eq:bound2}, we have
\begin{equation}\log D\ll \abs{d}A^2, \text{ and } \log\max(1, \abs{\sigma(\beta)})\ll \abs{d}A^2\end{equation} 
By \eqref{eq:not_a_one}, $D(\beta-1)$ is nonzero. Since it belongs to
$\Ocal_K$, its norm $\norm_{K/\QQ}(D(\beta-1))\in\ZZ$ is a nonzero integer. Therefore
\begin{equation}1\leq D^2\abs{\beta-1}\abs{\sigma(\beta)-1}\leq D^2\abs{\beta-1}\abs{\sigma(\beta)+1}.
\end{equation}
This gives 
\begin{equation}
    \abs{\beta-1}\geq \frac{1}{D^2\abs{\sigma(\beta)+1}}\gg e^{-C\abs{d}A^2} 
\end{equation} for some $C$ depending only on $X$.
By \eqref{eq:determinant_calculation_2} and the definition of $\beta$, we have
\begin{equation}\abs{\det(\lambda_{u\vert v},d)}=2\abs{\log\beta}.
\end{equation}
If $\abs{\log\beta}\geq1/2$, then the determinant is at least $1$, as required. If $\abs{\log\beta}\leq1/2$, the mean value theorem gives
$\abs{e^{\log\beta}-1}\leq e^{1/2}\abs{\log\beta}$, so
\begin{equation}
\abs{\det(\lambda_{u\vert v},d)}\geq 2e^{-1/2}\abs{\beta-1}\gg e^{-C\abs{d}A^2}.
\end{equation}
Since $\abs{d}A^2\geq1$, enlarging $M_1$ absorbs the implied constant and proves the claim.
\end{proof}
Recall the notation from Section \ref{sec:support_and_normal_turn}: $\Gamma_{H,j}$ is a connected component of $\partial B_X\backslash C_H$, and $s_{H,j}$ is its Euclidean length. The arc $J_{H,j}\subseteq S^1$ consists of the directions of supporting normals along $\Gamma_{H,j}$. For a small parameter $\eta>0$ we let $J_{H,j}^\eta$ denote the $\eta$-neighborhood of $J_{H,j}$ in $S^1$. It follows from Theorem \ref{thm:DoanLiNguyen} \eqref{thm:DoanLiNguyen_uno} that for $H\geq 1$ 
\begin{equation}
  \abs{J_{H,j}^\eta}\leq M e^{-\kappa H}+2\eta
\end{equation}

Fix $\varepsilon>0$ and assume that the parameters $H$ and $\eta$ are chosen such that \begin{equation}\label{eq:basic_assumption_1}m e^{-\Kappa H}\geq (2+\varepsilon)\eta\end{equation} Then, by Theorem \ref{thm:DoanLiNguyen} \eqref{thm:DoanLiNguyen_duo}, the arcs $J_{H,j}^\eta$ are disjoint, with gaps of at least $\varepsilon\eta$.

An integral covector $d\in (\ZZ^2)^*$ is called primitive if its dual vector $d^\#$ has coprime coordinates i.e. $d^\#\in \ZZ^2_{\mathrm{prim}}$. The covector $d$ is called \emph{unresolved} (with parameters $K$, $H$, $\eta$)  if 
\begin{equation}
  \frac{d^\#}{\abs{d^\#}}\in \cup_j J_{H,j}^\eta \text{ and } \abs{d}\leq K.
\end{equation} 
Since the arcs $J_{H,j}^\eta$ are disjoint, an unresolved covector $d$ satisfies $\tfrac{d^\#}{\abs{d^\#}}\in J_{H,j(d)}^\eta$ for a unique $j(d)$. 

\begin{lemma} \label{lem:unresolved_sizes} 
  Let $X$ be the modular punctured torus and $H\ge1$. Assume \eqref{eq:basic_assumption_1} and that $d$ is a nonzero integral unresolved covector with parameters $K$, $H$, $\eta$. If $A$ is the smaller of the two endpoint heights of $\Gamma_{H,j(d)}$, then
  \begin{equation}
    \abs{d} A^2\gg H, \quad s_{H,j(d)}\ll H^{-\tfrac 3 2 }\abs{d}^{\tfrac1 2}
  \end{equation}
\end{lemma}
\begin{proof}
Make the same coordinate-sector reduction as in the proof of  Lemma~\ref{lem:arc_length}.
Put $j=j(d)$, and let $u,w$ be the primitive representatives of the
endpoints of $\Gamma_{H,j}$, with $\infnorm{u}=A\le\infnorm{w}$.
For a signed coordinate permutation $T$ carrying their sector to $\mathcal C$,
write
\begin{equation}
 \begin{gathered}
 Tu=(p,q),\qquad Tw=(p',q'),\\
 0\le q\le p,\quad 0\le q'\le p',\quad 1\le p=A\le p'\le H.
 \end{gathered}
\end{equation}
For $z=(a,b)\in \ZZ^2$ with $a\neq 0$ we call $b/a$ its slope. It is the slope of the line between $0$ and $z$.  
The slopes $q/p$ and $q'/p'$ are consecutive Farey fractions, so $\abs{\det(u,w)}=1$.

To replace $w$ by a vector $v$ such that $u,v$ satisfy the hypotheses of Lemma~\ref{lem:support_separation}, set
\begin{equation}
 k=\left\lfloor\frac{p'}p\right\rfloor,\qquad v=w-ku.
\end{equation}
Write $Tv=(p_0,q_0)$. This vector need not lie in $\mathcal C$; the choice of $k$ gives $0\le p_0\le p-1$, and $\abs{pq_0-qp_0}=1$. Hence
\begin{equation}
 \abs{q_0}\le\frac{qp_0+1}{p}\le p_0+\frac1p\le p.
\end{equation}
Thus $v$ is primitive, $\abs{\det(u,v)}=1$, and $\infnorm{v}\le A$, so we can apply Lemma~\ref{lem:support_separation}. 
Every supporting functional $\mu$ at $u/\MRnorm{u}$ satisfies
$\mu(v)=\mu(w)-k\MRnorm{u}$. Since the subtracted term does not depend on $\mu$, the maximising functional is unchanged when $w$ is replaced by $v$, i.e. 
\begin{equation}\lambda_{u\vert w}=\lambda_{u\vert v}.\end{equation}
Applying Lemma~\ref{lem:support_separation}
to $u,v$ gives the lower bound
\begin{equation}\label{eq:unresolved_det_lower}
 \abs{\det(\lambda_{u\vert v},d)}\ge e^{-M_1\abs{d}A^2}.
\end{equation}

For an upper bound, we first show that $n(\lambda_{u\vert v})\in\overline{J_{H,j}}$.
For $t>0$, the slope 
\begin{equation}\frac{q+tq'}{p+tp'}=\left(\frac{p}{p+tp'}\right)\frac{q}{p}+\left(\frac{tp'}{p+tp'}\right)\frac{q'}{p'}\end{equation} of $T(u+tw)$ lies strictly between the consecutive endpoint slopes $q/p$ and $q'/p'$.
Thus $(u+tw)/\MRnorm{u+tw}$ lies on $\Gamma_{H,j}$.
Any supporting functional $\lambda_t$ at $(u+tw)/\MRnorm{u+tw}$ satisfies
\begin{equation}
 \lambda_t(w)
 =\frac{\MRnorm{u+tw}-\lambda_t(u)}{t}
 \ge\frac{\MRnorm{u+tw}-\MRnorm{u}}{t}
 \ge\lambda_{u\vert v}(w).
\end{equation}
Supporting functionals are uniformly bounded by \eqref{eq:equivalence_of_norms}.
Since $\lambda_t(u+tw)=\MRnorm{u+tw}$ and $\lambda_t(w)$ is bounded,
$\lambda_t(u)\to\MRnorm{u}$. Thus every subsequential limit of $\lambda_t$ as $t\to0^+$ supports at
$u/\MRnorm{u}$ and, by the displayed inequality, maximizes the value on $w$.
By uniqueness it therefore equals $\lambda_{u\vert w}=\lambda_{u\vert v}$.
Thus $\lambda_t\to\lambda_{u\vert v}$ and $n(\lambda_{u\vert v})\in\overline{J_{H,j}}$.

Since $d$ is unresolved, Theorem~\ref{thm:DoanLiNguyen} and
\eqref{eq:basic_assumption_1} now give that the angular distance between $n(\lambda_{u\vert v})$ and $d^\#/|d|$ is at most
\begin{equation}
\abs{J_{H,j}}+\eta\ll e^{-cH},
\end{equation}
where $c=\min(\kappa,\Kappa)>0$, since
$\eta\le\frac{m}{2+\varepsilon}e^{-\Kappa H}\le\frac m2e^{-\Kappa H}$.
The determinant is bounded by the product of the lengths and the angular
distance. Since $\abs{\lambda_{u\vert v}}\ll1$, it follows that
\begin{equation}
 \abs{\det(\lambda_{u\vert v},d)}\le L\abs{d}e^{-cH}
\end{equation}
for some $L\ge1$ depending only on $X$.
Comparing with \eqref{eq:unresolved_det_lower} and taking logarithms yields
\begin{equation}
 cH\le M_1\abs{d}A^2+\log\bigl(L\abs{d}\bigr)
 \ll\abs{d}A^2,
\end{equation}
because $A,\abs{d}\ge1$ and $\log\abs{d}\le\abs{d}$.
This proves the first assertion. Finally, $A\gg(H/\abs{d})^{1/2}$
and Lemma~\ref{lem:arc_length} give
\begin{equation}
 s_{H,j}
 \ll\frac{1}{HA}
 \ll H^{-3/2}\abs{d}^{1/2}.
\end{equation}
All implied constants depend only on $X$.
\end{proof}
Let $1\leq T\leq K$, and let 
$M_H(T)$ be the number of unresolved primitive covectors with $\abs{d}\leq T$.
\begin{lemma}\label{lemma:counting-unresolved}
Let $X$ be any complete finite-area hyperbolic once-punctured torus.
Assume \eqref{eq:basic_assumption_1} and
\begin{equation}\label{eq:basic_assumption_2}
H\ge D_X\log(2K),
\end{equation}
where $D_X$ is a sufficiently large constant depending only on $X$. Then
  \begin{equation}M_H(T)\ll \log^2(2T)\end{equation}
where the implied constant depends only on $X$.
\end{lemma}
\begin{proof}
Put $h=D_X\log(2T)$, choosing $D_X$ so that $D_X\log2\ge1$; then $1\le h\le H$.
Recall the definition \eqref{eq:corners_definition} of $C_H$. Since $C_h\subseteq C_H$, each $\Gamma_{H,j}$ is contained in some
$\Gamma_{h,i}$, and consequently $J_{H,j}^\eta\subseteq J_{h,i}^\eta$.
Thus every unresolved direction belongs to some $J_{h,i}^\eta$.
Put $c=\min(\kappa,\Kappa)>0$. Theorem~\ref{thm:DoanLiNguyen} and
\eqref{eq:basic_assumption_1}, together with $h\le H$, give
\begin{equation}
 \abs{J_{h,i}^\eta}
 \le M e^{-\kappa h}+m e^{-\Kappa h}
 \le (M+m)(2T)^{-cD_X}
 <T^{-2}.
\end{equation}
Here we used $2\eta\le m e^{-\Kappa h}$; the final inequality holds for $D_X$
sufficiently large depending only on $X$, uniformly in $\varepsilon$.

Suppose one such arc contains the directions of two distinct primitive integral
covectors $d,e$ with $\abs{d},\abs{e}\le T$.
Let $\alpha\in[0,\pi]$ be the angle between their dual unit vectors
$d^\#/\abs{d}$ and $e^\#/\abs{e}$.
Primitivity excludes identical directions, while
$\alpha<T^{-2}\le1<\pi$ excludes opposite directions.
Thus $d,e$ are nonparallel, so
\begin{equation}
 1\le\abs{\det(d,e)}
 \le\abs{d}\abs{e}\alpha<T\cdot T\cdot T^{-2}=1,
\end{equation}
a contradiction. Hence each arc contains at most one such direction.
There are $\#C_h\ll h^2$ arcs by Lemma~\ref{lem:arc_length}, so
\begin{equation}
 M_H(T)\le\#C_h\ll h^2\ll\log^2(2T).\qedhere
\end{equation}
\end{proof}
For the modular torus, under these assumptions, summation by parts together with
Lemmas~\ref{lem:unresolved_sizes} and~\ref{lemma:counting-unresolved} gives
\begin{equation}\label{eq:unresolved_bound}
  \sum_{\substack{d \text{ primitive and unresolved}\\ \abs{d}\leq K}}\frac{s_{H,j(d)}}{\abs{d}}\ll H^{-3/2}.
\end{equation}

\section{Convex caps}
We define the support function of $B_X$ by
\begin{equation}\label{eq:support_function}
\begin{array}{@{}r@{\;}l@{\;}c@{\;}c@{}}
h_{B_X}: & S^1 & \to & \RR_+ \\
     & \theta & \to & \displaystyle\max_{z\in B_X}z\cdot \theta
\end{array}
\end{equation}
This gives the perpendicular distance from the origin to the supporting line of $B_X$ with outward unit normal $\theta$. Note also that it gives the positive endpoint  $h_{B_X}(\theta)\theta$ of the projection of $B_X$ to the line spanned by $\theta$.  The line $z\cdot \theta=h_{B_X}(\theta)$ is a supporting line at some $w\in \partial B_X$ i.e. $z\cdot \theta\leq h_{B_X}(\theta)$ for $z\in B_X$ with equality at $w$. For the basics about the support function consult \cite[Sec 1]{Schneider:1998}.

\noindent
\begin{minipage}[t]{.49\linewidth}
\vspace{0pt}
Define the convex cap of width $\delta>0$ in direction $\theta\in S^1$ by 
\begin{equation}
  C(\theta, \delta):=\{z\in B_X:z\cdot \theta \geq h_{B_X}(\theta)-\delta\}
\end{equation}
and its diameter by
\begin{equation}
\diam(C(\theta,\delta))
\end{equation} i.e. the maximal distance between two points in $C(\theta,\delta)$.
\end{minipage}\hfill
\begin{minipage}[t]{.48\linewidth}
\vspace{0pt}
\centering
\makeatletter\def\@captype{figure}\makeatother
\scalebox{1.05}{%
\begin{tikzpicture}[x=2.5075cm,y=2.5075cm]
  \def\CapDelta{0.1}
  \begin{scope}
    \clip (-.7,-.1) rectangle (.7,.75);
    \filldraw[fill=FigureBlue!18,draw=FigureBlue,line width=.55pt]
      \MRBallBoundary;
  \end{scope}
  \coordinate (CapTop) at (MRball-1-2);
  \path (CapTop) ++(0,-\CapDelta) coordinate (CapCut);
  \begin{scope}
    \clip (-.7,0 |- CapCut) rectangle (.7,.75);
    \filldraw[fill=orange!55,draw=FigureBlue,line width=.55pt]
      \MRBallBoundary;
  \end{scope}
  \draw[thin,black!55,-{Stealth[length=3pt]}] (-.69,0) -- (.88,0)
    node[below right,font=\scriptsize] {$x$};
  \draw[thin,black!55,-{Stealth[length=3pt]}] (0,-.1) -- (0,1.12)
    node[above left,font=\scriptsize] {$y$};
  \draw[thin,black!55] (-.02,1) -- (0,1)
    node[pos=0,left=2pt,font=\scriptsize] {$1$};
  \fill[black] (0,0) circle[radius=.7pt];
  \node[below left=2pt,font=\scriptsize] at (0,0) {$0$};
  \node[font=\scriptsize] at (.28,.18) {$B_X$};
  \draw[black,line width=.7pt]
    (-.7,0 |- CapTop) -- (.88,0 |- CapTop);
  \node[anchor=south west,inner sep=1pt,yshift=1.5pt,font=\scriptsize,text=black]
    at (-1.23,0 |- CapTop) {$z\cdot\theta=h_{B_X}(\theta)$};
  \draw[black,line width=.7pt]
    (-.7,0 |- CapCut) -- (.88,0 |- CapCut);
  \node[anchor=north west,inner sep=1pt,yshift=-1.5pt,font=\scriptsize,text=black]
    at (-1.23,0 |- CapCut) {$z\cdot\theta=h_{B_X}(\theta)-\delta$};
  \draw[black,{Stealth[length=2pt]}-{Stealth[length=2pt]},line width=.6pt]
    (.73,0 |- CapCut) -- (.73,0 |- CapTop)
    node[midway,right=4pt,fill=white,inner sep=1pt,font=\scriptsize]
      {$\delta$};
  \draw[FigureBlue,-{Stealth[length=3pt]},line width=.8pt]
    (0,0) -- (0,1)
    node[pos=.85,right=3pt,font=\scriptsize,text=black] {$\theta$};
  \fill[FigureBlue] (CapTop) circle[radius=1pt];
  \node[above =1.5pt,font=\scriptsize,text=black] at (CapTop) {$w$};
  \draw[thin,black!55] (CapTop) ++(0,{-\CapDelta/2}) -- (.45,.73)
    node[above,font=\scriptsize,text=black] {$C(\theta,\delta)$};
\end{tikzpicture}%
}
{\small\caption{$C(\theta,\delta)$ of $B_X$, 
$\theta=(0,1)$.}}
\label{fig:convex-cap-vertical}
\end{minipage}

\begin{lemma}\label{lem:cap_diameter}
Let $H\ge1$, and assume \eqref{eq:basic_assumption_1}.
Then for every $\delta>0$ and $\theta\in S^1$,
\begin{equation}
\diam(C(\theta,\delta))\ll
\begin{cases}
 \delta/\eta & \text{if } \theta\notin\bigcup_j J_{H,j}^\eta,\\
 s_{H,j}+\delta/\eta & \text{if } \theta\in J_{H,j}^\eta.
\end{cases}
\end{equation}
\end{lemma}
\begin{proof}
Since $C_H$ contains at least eight corners and the total turn is $2\pi$,
Theorem~\ref{thm:DoanLiNguyen} and \eqref{eq:basic_assumption_1} give
\begin{equation}
 (2+\varepsilon)\eta\le m e^{-\Kappa H}\le\frac{2\pi}{8},
\end{equation}
so $0<\eta<\pi/8$.
Let $e\perp\theta$ be a unit vector, and let $a_\pm\in\partial B_X$ be the
unique supporting points with outward unit normals
\begin{equation}
 n_\pm=\cos\eta\,\theta\pm\sin\eta\,e.
\end{equation}
As discussed after \eqref{eq:basic_assumption_1},
Theorem~\ref{thm:DoanLiNguyen}(\ref{thm:DoanLiNguyen_duo}) gives normal gaps
between consecutive $J_{H,j}$ of length at least $m e^{-\Kappa H}\ge(2+\varepsilon)\eta>2\eta$.
The interval of normals between $n_-$ and $n_+$ has angular length $2\eta$ and therefore meets at most one such arc.
If $\theta\notin\bigcup_j J_{H,j}^\eta$, both supporting points are the same corner,
so $a_+=a_-$.
If $\theta\in J_{H,j}^\eta$, both points belong to $\overline{\Gamma_{H,j}}$,
so $\abs{a_+-a_-}\le s_{H,j}$.

For any $z\in C(\theta,\delta)$, the supporting inequalities and the cap definition give
\begin{equation}
 (z-a_\pm)\cdot n_\pm\le0,
 \qquad (z-a_\pm)\cdot\theta\ge-\delta.
\end{equation}
Rearranging yields
\begin{equation}
 a_-\cdot e-\delta\cot\eta
 \le z\cdot e
 \le a_+\cdot e+\delta\cot\eta.
\end{equation}
Thus the cap has width at most $\abs{a_+-a_-}+2\delta\cot\eta$ in direction $e$,
and at most $\delta$ in direction $\theta$. Therefore
\begin{equation}
 \diam(C(\theta,\delta))
 \le\abs{a_+-a_-}+2\delta\cot\eta+\delta
 \ll\abs{a_+-a_-}+\frac{\delta}{\eta}.
\end{equation}
Substituting the two bounds for $\abs{a_+-a_-}$ proves the claim.
\end{proof}

\section{Smoothing and Poisson summation} 
We are now ready to start proving Theorem \ref{thm:main}. The first step is to find an expansion of $N_X(R)$ with an error estimate which is better than the Jarn\'{\i}k-Steinhaus error term $O(R)$, see \eqref{eq:convex_bound_JS}. To do so we use the Poisson summation formula (see \cite[Sec. 4.3]{IwaniecKowalski:2004}). As usual it is convenient to introduce a smooth approximation of the characteristic function of the McShane-Rivin ball. 

 Let $\rho:\RR^2\to \RR$ be a non-negative smooth radially symmetric function supported in the Euclidean unit ball, normalized to have total mass 1. For $Q\geq 1$ consider $\rho_Q(x)=Q^2\rho(Q x)$, which is supported in the Euclidean ball of radius $1/Q$ and has total mass 1. 
For sufficiently nice functions $f$ and $g$ on $\RR^2$ we define their convolution 
\begin{equation}f*g(x)=\int_{\RR^2}f(y)g(x-y)dy.\end{equation}  

Then for every $S>0$ and $Q\geq 1$ the function $1_{SB_X}*\rho_Q$ is smooth: It  follows from \eqref{eq:equivalence_of_norms} and the triangle inequality for the McShane-Rivin norm that 
\begin{equation}1_{SB_X}*\rho_Q(x)=\begin{cases}1&\text{ if }\MRnorm{x}\leq S-C/Q \\ 0 &\text{ if }\MRnorm{x}>S+C/Q,  \end{cases}\end{equation} and $0\leq 1_{SB_X}*\rho_Q(x)\leq 1$ for all $x\in \RR^2$. Here $C$ is the constant from \eqref{eq:equivalence_of_norms}. 
Assume
\begin{equation}\label{eq:basic_assumption_3}
Q\geq 1,\qquad R>C/Q.
\end{equation}
Define the functions
\begin{equation}1^\pm_{RB_X,Q}(x)=\bigl(1_{(R\pm C/Q)B_X} *\rho_Q\bigr)(x)\end{equation}
The preceding inequalities give \begin{equation}\label{eq:inequalities}1^-_{RB_X,Q}(x)\leq 1_{RB_X}(x)\leq 1^+_{RB_X,Q}(x).\end{equation}
For $f\in L^1(\RR^2)$ its Fourier transform is defined by $\hat f(y)=\int_{\RR^2}f(x)e^{-2\pi i x\cdot y}dx$. The Fourier transform of a convolution is the product of the Fourier transforms, i.e. $\widehat{f*g}=\hat f\hat g$.
We use Poisson summation on $1^\pm_{RB_X,Q}$ and find that 
\begin{equation}\label{eq:Poisson}\sum_{n\in \ZZ^2}1^\pm_{RB_X,Q}(n)=\sum_{n\in \ZZ^2}\widehat{1^\pm_{RB_X,Q}}(n)=\sum_{n\in \ZZ^2}\widehat{1_{(R\pm C/Q)B_X}}(n)\widehat{\rho_Q}(n)\end{equation}

For notational simplicity let $S=R\pm C/Q$. The term corresponding to $n=0$ on the Fourier side gives the main term $\vol(B_X)S^2$, leaving the sum
\begin{equation}\sum_{0\neq n\in \ZZ^2}\widehat{1_{SB_X}}(n)\widehat{\rho_Q}(n)=S^2\sum_{0\neq n\in\ZZ^2}\widehat{1_{B_X}}(Sn)\widehat{\rho}(n/Q).\end{equation}

The strategy is to split the sum as follows: For $0\neq n\in \ZZ^2$ we write $n$ uniquely $n=km$ where $k\in \ZZ^2_{\mathrm{prim}}$, and $m\in \NN$.  We split the sum into two parts according to $\abs{k}\leq K$ and $\abs{k}> K$. Here $K>Q$ is a parameter to be chosen later.  
The first part is 
\begin{equation}
S^2\sum_{\substack{k \textrm{primitive}\\\abs{k}\leq K}} \sum_{m\in \NN} \widehat{1_{B_X}}(Smk)\widehat{\rho}(mk/Q)=\sum_{\substack{k \textrm{primitive}\\\abs{k}\leq K}}M_k(S,Q).\end{equation}
A crucial part in proving our main theorem is to use our  understanding of the geometry of the McShane-Rivin ball to get a good estimate on $M_k(S,Q)$. 

The remaining terms on the right-hand side of \eqref{eq:Poisson} are easier to estimate: their total contribution is bounded by
\begin{equation}S^2\sum_{\abs{n}> K}  \abs{\widehat{1_{B_X}}(Sn)}\abs{\widehat{\rho}(n/Q)}\ll S \sum_{\abs{n}>K}\frac{Q^A}{\abs{n}^{A+1}}\ll S Q^AK^{-(A-1)} 
\end{equation}
for integers $A$ satisfying
\begin{equation}\label{eq:basic_assumption_4}
A>1.
\end{equation}
Here we have used that for $\abs{y}>1$
\begin{equation}
    \abs{\widehat{\rho}(y)}\ll \frac{1} {\abs{y}^A}, \quad \abs{\widehat{1_{B_X}}(y)}\ll \frac{1}{\abs{y}}. 
\end{equation}
The first bound follows by integrating by parts $A$ times in the direction $y/\abs{y}$. For the second, applying the Gauss--Green theorem to the vector field $y e^{-2\pi i x\cdot y}$ expresses $2\pi i\abs{y}^2\widehat{1_{B_X}}(y)$ as a boundary integral of magnitude at most $\abs{y}$ times the perimeter of $B_X$. This gives the stated bound since $B_X$ has finite perimeter (compare \cite[Lemma 2.14]{KoivusaloLagace:2025}).

Assume
\begin{equation}\label{eq:basic_assumption_5}
  Q\to \infty\text{ as }R\to\infty,\qquad K>Q.
\end{equation}
Then
\begin{equation}\label{eq:S_pm-expansion}\sum_{n\in \ZZ^2}1^\pm_{RB_X,Q}(n)=\vol(B_X)R^2+\sum_{\substack{k \textrm{primitive}\\\abs{k}\leq K}}M_k(S,Q)+O(1+R/Q+ R Q^AK^{-(A-1)}).
\end{equation}
This reduces the problem to estimating $M_k(S,Q)$ which we do in the next section.
\section{Bounding Fourier coefficients on a radial line} In the previous section we considered 
\begin{equation}M_k(S,Q)=S^2\sum_{m\in \NN} \widehat{1_{B_X}}(Smk)\widehat{\rho}(mk/Q). \end{equation} We first need a result about sampling concave functions.
\begin{lemma}\label{lem:concave_sampling}
Let $f:\RR\to\RR$ be continuous and nonnegative, with $f\vert_{[0,1]}$ concave
and $f\vert_{[0,1]^c}=0$. Let $0<h<\tfrac18$. Then for every $y\in\RR$,
\begin{equation}
\abs{h\sum_{n\in\ZZ}f(h(n+y))-\int_\RR f(x)\,dx}\le 7 h\bigl(f(h)+f(1-h)\bigr).
\end{equation}
\end{lemma}

\begin{proof} By continuity $f(0)=f(1)=0$. Let $0\leq s\leq h$ and consider $h=\alpha s+(1-\alpha)\cdot 1$ with
$\alpha=\frac{1-h}{1-s}\in[0,1]$. Concavity gives
\begin{equation}
f(h)= f(\alpha s+(1-\alpha) 1)\geq \alpha f(s)\geq (1-h)f(s)\geq \tfrac{7}{8}f(s)\geq \tfrac 1 2 f(s).
\end{equation}
If $h\leq s\leq 2h$ then $h=\tfrac{h}{s}s$ so $f(h)=f(\tfrac{h}{s}s+(1-\tfrac{h}{s})0)\geq \tfrac{h}{s}f(s)\geq \tfrac{1}{2}f(s)$
so 
\begin{equation}\label{eq:first_end_bound}
    f(s)\leq  2 f(h)\textrm{ for }0\leq s\leq 2h.\end{equation} Applying this to $x\mapsto f(1-x)$ gives
\begin{equation}\label{eq:second_end_bound}
f(s)\leq  2  f(1-h) \text{ for } 1-2h \leq  s \leq 1. 
\end{equation}

Let $x_n=h(n+y)$. There are at most finitely  many (and at least 8) $n$ with $x_n\in [0,1]$. If this happens for $N_0\leq n \leq N_1$ then $0\leq x_{N_0}\leq h$ and $1-h\leq x_{N_1}\leq 1$.  
 
For $N_0< n< N_1$ and $x\in [x_n,x_{n+1}]$ concavity implies that $f(x)$ lies above the secant line of the endpoints,
 and below the secant line of the previous interval. Concretely if \begin{equation}a_n=\frac{f(x_{n+1})-f(x_n)}{x_{n+1}-x_n}=\frac{f(x_{n+1})-f(x_n)}{h}\end{equation} then 
\begin{equation}
 a_{n}(x-x_{{n}})+f(x_{n})   \leq f(x)\leq 
a_{n-1}(x-x_{{n}})+f(x_{n})
\end{equation}
so that $a_{n}(x-x_n)\leq f(x)-f(x_n)\leq  a_{n-1}(x-x_n)$.

Integrating from $x_{N_0+1}$ to $x_{N_1}$ gives \begin{equation}
    \sum_{n=N_0+1}^{N_1-1}\frac{a_nh^2}{2}\leq  \int_{x_{N_0+1}}^{x_{N_1}}f(x)dx-h\sum_{n=N_0+1}^{N_1-1}f(x_n)\leq \sum_{n=N_0+1 }^{N_1-1}\frac{a_{n-1}h^2}{2}
\end{equation}

The first and the last sum are telescoping giving $\tfrac{h}{2}$ times either $f(x_{N_1})-f(x_{N_0+1})$ or $f(x_{N_1-1})-f(x_{N_0})$ so both sums are bounded in absolute value by 
$\tfrac{h}{2}2(f(h)+f(1-h))$. Noting that  
    \begin{equation}h\sum_{n\in \ZZ}f(x_n)-h\sum_{n=N_0+1}^{N_1-1}f(x_n)=h(f(x_{N_0})+f(x_{N_1}))\end{equation}
and 
    \begin{equation}\int_{\RR}f(x)dx -\int_{x_{N_{0}+1}}^{x_{N_1}}f(x)dx= \int_0^{x_{N_{0}+1}}f(x)dx+\int_{x_{N_1}}^1f(x)dx \end{equation}
the bound follows from \eqref{eq:first_end_bound} and \eqref{eq:second_end_bound} upon observing that $0\leq x_{N_0},x_{N_0+1}\leq 2h$ and $1-2h\leq x_{N_1-1},x_{N_1}\leq 1$.
\end{proof}

The constant 7 is somewhat arbitrary and can be improved at the cost of making the proof a bit longer. 

\begin{lemma}\label{lem:radial_fourier_sum}
  For $S$ sufficiently large, $Q\geq 1$ and $k\in \ZZ^2_{\mathrm{prim}}$, we have 

  \begin{equation}
    \abs{M_k(S,Q)}\ll\frac{S}{\abs{k}}\diam\left(C\left(\frac{k}{\abs{k}},\frac{1}{S\abs{k}}\right)\right).
  \end{equation}
Assume furthermore that $H\geq 1$, $\abs{k}\leq K$, and \eqref{eq:basic_assumption_1}. Let $d\in (\ZZ^2)^*$ be the integer functional satisfying $d^\#=k$. Then 
\begin{equation}
  \abs{M_k(S,Q)}\ll \frac{1}{\eta\abs{k}^2}+
  \begin{cases}
    S\dfrac{\,s_{H,j(d)}}{\abs{k}} & \text{if }d\text{ is unresolved with parameters $K$, $H$, $\eta$.},\\[3pt]
    0 & \text{otherwise}.
  \end{cases}
\end{equation}
\end{lemma}
\begin{proof}
Put $\theta=k/\abs{k}$ and $h=(S\abs{k})^{-1}$. Choose a unit vector
$e\perp\theta$ and define
\begin{equation}
 f(t)=\int_{\RR}1_{B_X}(t\theta+se)\,ds.
\end{equation}
Since $e$ is a unit vector, $f(t)$ is the one-dimensional Euclidean length of
$B_X\cap\{x:x\cdot\theta=t\}$. Let $[a,b]$ be the projection of $B_X$
onto the $\theta$-axis. For $t\in[a,b]$, write
\begin{equation}
 B_X\cap\{x:x\cdot\theta=t\}
 =\{t\theta+se:\ell(t)\le s\le r(t)\}.
\end{equation}
Convexity of $B_X$ makes $\ell$ convex and $r$ concave.
Hence $f=r-\ell$ is concave on $[a,b]$ (This is the planar slice case of Brunn-Minkowski \cite[Thm. 7.1.1]{Schneider:2014}) and therefore continuous on $(a,b)$ (\cite[Thm. 2.2]{Gruber:2007}).
Strict convexity makes each supporting slice at $a$ and $b$ a single point.
By compactness, nearby slice lengths tend to zero: otherwise their endpoints
would have two distinct limit points on the same supporting line.
Thus $f(a)=f(b)=0$, and $f$, which is zero outside $[a,b]$, is continuous
on $\RR$. Finally, $B_X=-B_X$ gives $f(-t)=f(t)$.

Define the smoothed slice function
\begin{equation}
 g(t)=\int_{\RR^2}\rho_{SQ}(y)f(t-y\cdot\theta)\,dy.
\end{equation}
Here $\rho_{SQ}(y)=(SQ)^2\rho(SQy)$. The function $g$ is smooth and compactly
supported, and Fubini's theorem together with Fourier scaling gives
\begin{equation}
 \widehat g(m/h)=\widehat{1_{B_X}}(Smk)\widehat\rho(mk/Q).
\end{equation}
Moreover, $g$ is even because $f$ and $\rho_{SQ}$ are even. Thus
one-dimensional Poisson summation yields
\begin{equation}\label{eq:poisson_slice_discrepancy}
\begin{aligned}
 \frac{2M_k(S,Q)}{S^2}
 &=h\sum_{n\in\ZZ}g(nh)-\int_\RR g(t)\,dt\\
 &=\int_{\RR^2}\rho_{SQ}(y)\left[
 h\sum_{n\in\ZZ}f(nh-y\cdot\theta)
 -\int_\RR f(t)\,dt\right]dy.
\end{aligned}
\end{equation}
For sufficiently large $S$, we have $h/(b-a)<1/8$ uniformly in $k$, since
$b-a$ is bounded below and $\abs{k}\ge1$. Rescaling
Lemma~\ref{lem:concave_sampling} from $[0,1]$ to $[a,b]$ therefore gives,
for every $u\in\RR$,
\begin{equation}\label{eq:shifted_slice_discrepancy}
 \abs{h\sum_{n\in\ZZ}f(nh-u)-\int_\RR f(t)\,dt}
 \le 7h\bigl(f(a+h)+f(b-h)\bigr).
\end{equation}
For each $y\in\RR^2$, the expression in square brackets in
\eqref{eq:poisson_slice_discrepancy} is the quantity inside the absolute value
in \eqref{eq:shifted_slice_discrepancy} with $u=y\cdot\theta$.
Since $\rho_{SQ}\ge0$ and $\int_{\RR^2}\rho_{SQ}=1$,
\begin{equation}
 \abs{M_k(S,Q)}
 \le\frac{7S}{2\abs{k}}\bigl(f(a+h)+f(b-h)\bigr).
\end{equation}
The two slices lie in $C(-\theta,h)$ and $C(\theta,h)$, respectively,
so their lengths are bounded by the corresponding cap diameters.
Since $B_X=-B_X$, we have $C(-\theta,h)=-C(\theta,h)$, and hence these
diameters are equal.
Consequently,
\begin{equation}
 \abs{M_k(S,Q)}
 \ll\frac{S}{\abs{k}}
 \diam\left(C\!\left(\theta,\frac{1}{S\abs{k}}\right)\right).
\end{equation}
For the second estimate, since $\abs{k}\le K$, the covector $d$ is unresolved precisely
when $\theta\in\bigcup_j J_{H,j}^\eta$. Applying
Lemma~\ref{lem:cap_diameter} to the preceding bound and using
$Sh/(\eta\abs{k})=1/(\eta\abs{k}^2)$ proves the claim.
\end{proof}

\begin{cor}\label{cor:colinear_bound}
Let $X$ be the hyperbolic modular punctured torus. Let $S,H$ be sufficiently
large, let $K,Q\geq1$, and assume \eqref{eq:basic_assumption_1}. Then
\begin{equation}\label{eq:colinear_bound}
 \sum_{\substack{k\in\ZZ^2_{\mathrm{prim}}\\\abs{k}\le K}}
 \abs{M_k(S,Q)}
 \ll \frac{\log(2K)}{\eta}+SH^{-3/2}+SK e^{-cH},
\end{equation}
where $c=\min(\kappa,\Kappa)>0$, with $\kappa,\Kappa$ as in
Theorem~\ref{thm:DoanLiNguyen}. The implied constant depends only on $X$.
\end{cor}
\begin{proof}
By Lemma~\ref{lem:radial_fourier_sum} and the elementary lattice estimate
$\sum_{0<\abs{k}\le K}\abs{k}^{-2}\ll\log(2K)$, it suffices to show that
\begin{equation}\label{eq:unresolved_weighted_general}
 \sum_{\substack{d\text{ primitive and unresolved}\\\abs{d}\le K}}
 \frac{s_{H,j(d)}}{\abs{d}}
 \ll H^{-3/2}+K e^{-cH}.
\end{equation}
Put $L=H^{3/2}$. For sufficiently large $H$ we have
$H\ge D_X\log(2L)$, so \eqref{eq:unresolved_bound}, applied with
cutoff $\min(K,L)$, bounds the contribution from $\abs{d}\le L$ by
$O(H^{-3/2})$. If $K\le L$, this proves the claim.

Suppose $K>L$. By Theorem~\ref{thm:DoanLiNguyen}\eqref{thm:DoanLiNguyen_uno}
and \eqref{eq:basic_assumption_1}, we have
$\abs{J_{H,j}^\eta}\ll e^{-cH}$. Set $\alpha_j=\abs{J_{H,j}^\eta}$, and let
$N_j(T)$ count primitive integral covectors $d$ with $\abs{d}\le T$ and
$d^\#/\abs{d}\in J_{H,j}^\eta$. The determinant argument in the proof of
Lemma~\ref{lemma:counting-unresolved} shows that distinct such directions
are separated by at least $T^{-2}$. Hence
\begin{equation}
 N_j(T)\ll1+\alpha_j T^2\ll1+T^2e^{-cH}.
\end{equation}
Partial summation now gives
\begin{equation}
\begin{aligned}
 \sum_{\substack{d\text{ primitive and unresolved}\\
                  L<\abs{d}\le K,\ j(d)=j}}\frac1{\abs{d}}
 &=\frac{N_j(K)}K-\frac{N_j(L)}L+\int_L^K\frac{N_j(t)}{t^2}\,dt\\
 &\le\frac{N_j(K)}K+\int_L^K\frac{N_j(t)}{t^2}\,dt
 \ll L^{-1}+K e^{-cH}.
\end{aligned}
\end{equation}
Multiplying by $s_{H,j}$ and summing over $j$, we use
$\sum_j s_{H,j}=\operatorname{length}(\partial B_X)$ to bound the remaining
contribution by $O(H^{-3/2}+K e^{-cH})$. Together with the first part,
this proves \eqref{eq:unresolved_weighted_general} and the corollary.
\end{proof}

\begin{thm}\label{thm:lattice_count}
Let $X$ be the hyperbolic modular punctured torus, and let
$N_X(R)=\#(RB_X\cap\ZZ^2)$ be the full lattice-point count. Then, as
$R\to\infty$,
\begin{equation}
 N_X(R)=\vol(B_X)R^2+
 O\!\left(\frac{R}{(\log R)^{3/2}}\right).
\end{equation}
\end{thm}
\begin{proof}
Recall that
\begin{equation}\label{eq:S_X-expression}
 N_X(R)=\sum_{n\in\ZZ^2}1_{RB_X}(n).
\end{equation}
The inequalities \eqref{eq:inequalities} bound this sum between the two
smoothed sums. For $S_\pm=R\pm C/Q$ we have $S_\pm\asymp R$ when $R$
is large. Thus \eqref{eq:S_pm-expansion}, with $A=3$, and
Corollary~\ref{cor:colinear_bound} give
\begin{equation}\label{eq:lattice_count_error}
\begin{split}
 \abs{N_X(R)-\vol(B_X)R^2}
 \ll {}&\frac{\log(2K)}\eta+RH^{-3/2}+RK e^{-cH}\\
       &+1+\frac RQ+\frac{RQ^3}{K^2},
\end{split}
\end{equation}
provided \eqref{eq:basic_assumption_1}, \eqref{eq:basic_assumption_2},
\eqref{eq:basic_assumption_3}, \eqref{eq:basic_assumption_4}, and
\eqref{eq:basic_assumption_5} hold and $S,H$ are sufficiently large.

Put $L=\log R$, and choose
\begin{equation}
 Q=L^2,\qquad K=L^4,\qquad
 H=\frac{L}{4\Kappa},\qquad \eta=R^{-1/2},
\end{equation}
where $\Kappa$ is the constant in \eqref{eq:basic_assumption_1}.
All the required conditions hold for sufficiently large $R$:
$Q\to\infty$, $K>Q\ge1$, $R>C/Q$, $A=3>1$, and
$H/\log(2K)\to\infty$, while $S_\pm,H\to\infty$. In particular,
\begin{equation}
 m e^{-\Kappa H}=mR^{-1/4}\ge(2+\varepsilon)R^{-1/2}.
\end{equation}
Writing $\delta=c/(4\Kappa)>0$, the right-hand side of
\eqref{eq:lattice_count_error} is therefore
\begin{equation}
 \ll R^{1/2}\log L+\frac{R}{L^{3/2}}
       +R^{1-\delta}L^4+1+\frac{R}{L^2}
 \ll\frac{R}{L^{3/2}}.
\end{equation}
This proves the theorem.
\end{proof}

The parameter choices above are not unique, but changing them cannot improve
the logarithmic exponent in \eqref{eq:lattice_count_error}. Indeed,
\eqref{eq:basic_assumption_1} and $K>Q\ge1$ give
$\log(2K)/\eta\gg e^{\Kappa H}$. Keeping this term of order at most
$R/(\log R)^{3/2}$ forces $H=O(\log R)$, while the term $RH^{-3/2}$ is then
at least of order $R/(\log R)^{3/2}$.

\section{Counting essential simple closed geodesics}
Theorem~\ref{thm:lattice_count} gives an improved error term for the full
lattice-point count. Apart from the origin, this count records the positive
iterates of oriented essential simple closed geodesics. We now pass to
primitive lattice points, which count the oriented essential simple geodesics themselves.
 This is done by an inclusion-exclusion argument on the result in Theorem~\ref{thm:lattice_count} using the M\"obius function. In order to control the error term in such an argument we need some control over the zeroes of the Riemann zeta function.

\begin{thm}\label{thm:primitive_count}
Assume that $\zeta(s)$ has no zeros in $\Re(s)>1-\delta_0$ for some fixed
$\delta_0>0$. Let $P_X(R)$ be the number of oriented simple
essential closed geodesics on the modular punctured torus $X$ of length
at most $R$. Then, as $R\to\infty$,
\begin{equation}
 P_X(R)=\frac{\vol(B_X)}{\zeta(2)}R^2+
 O\!\left(\frac{R}{(\log R)^{1/2}}\right).
\end{equation}
\end{thm}
\begin{proof}
Put $N_X^*(t)=N_X(t)-1$ and $A=\vol(B_X)$. Grouping nonzero lattice
points according to the greatest common divisor of their coordinates,
and then applying M\"obius inversion, gives
\begin{equation}
 N_X^*(R)=\sum_{g\ge1}P_X(R/g),
 \qquad
 P_X(R)=\sum_{g\ge1}\mu(g)N_X^*(R/g).
\end{equation}
Both sums are finite, since the norm of a nonzero lattice point is bounded
away from zero.

Choose $Y\ge e^2$ and put $D=R/Y\ge2$. For $g\le D$, Theorem~\ref{thm:lattice_count}
and an integral comparison give
\begin{align*}
 \sum_{g\le D}\mu(g)N_X^*(R/g)
 &=AR^2\sum_{g\le D}\frac{\mu(g)}{g^2}
   +O\!\left(R\sum_{g\le D}
       \frac{1}{g(\log(R/g))^{3/2}}\right)\\
 &=AR^2\sum_{g\le D}\frac{\mu(g)}{g^2}
   +O\!\left(\frac{R}{\sqrt{\log Y}}\right).
\end{align*}
For the remaining terms, introduce the Mertens function and put
\begin{equation}
 M_\mu(t)=\sum_{g\le t}\mu(g),
 \qquad
 \Delta(D)=\sup_{t\ge D}\frac{\abs{M_\mu(t)}}t.
\end{equation}
Interchanging the two finite sums then gives
\begin{equation}
 \sum_{g>D}\mu(g)N_X^*(R/g)
 =\sum_{\substack{0\neq v\in\ZZ^2\\ \MRnorm{v}\le Y}}
   \left(M_\mu\!\left(\frac{R}{\MRnorm{v}}\right)-M_\mu(D)\right).
\end{equation}
Norm equivalence gives $N_X^*(Y)\ll Y^2$ and
$\sum_{0<\MRnorm{v}\le Y}\MRnorm{v}^{-1}\ll Y$. The last display is
therefore $O(RY\Delta(D))$.
Partial summation also gives
\begin{equation}
 \abs{\sum_{g>D}\frac{\mu(g)}{g^2}}\ll\frac{\Delta(D)}D.
\end{equation}
Since $\sum_{g\ge1}\mu(g)/g^2=1/\zeta(2)$, we have proved the general estimate
\begin{equation}\label{eq:primitive_count_cutoff}
 P_X(R)=\frac{A}{\zeta(2)}R^2+
 O\!\left(\frac{R}{\sqrt{\log Y}}+RY\Delta(R/Y)\right).
\end{equation}

The fixed zero-free hypothesis implies
$M_\mu(t)\ll t^{1-\delta}$ for some $\delta>0$, by the Perron
contour shift used under RH in
\cite[\S13.2, p.~441]{MontgomeryVaughan:2007} and
\cite[p.~48]{Ivic:1985}, with the new line chosen inside the assumed
zero-free half-plane.
Thus $\Delta(D)\ll D^{-\delta}$. Choose $Y=R^a$ with a fixed
$0<a<\delta/(1+\delta)$. The second error in
\eqref{eq:primitive_count_cutoff} is then
$O(R^{1-\delta+a(1+\delta)})=o(R/\sqrt{\log R})$, while the first is
$O(R/\sqrt{\log R})$, as required.
\end{proof}

\begin{remark}
\label{rem:unconditional_count}
Without a fixed zero-free half-plane, the classical
Vinogradov--Korobov estimate gives
\begin{equation}
 M_\mu(t)\ll t e^{-c\omega(t)},\qquad
 \omega(t)=\frac{(\log t)^{3/5}}{(\log\log t)^{1/5}},
\end{equation}
for some $c>0$ and sufficiently large $t$; see
\cite[Chap.~V, \S5, Satz~3, p.~191]{Walfisz:1963}.
Consequently $\Delta(D)\ll e^{-c\omega(D)}$. Taking
$Y=\exp(c\omega(R)/4)$ in \eqref{eq:primitive_count_cutoff}, we have
$\omega(R/Y)\sim\omega(R)$, so
$Y\Delta(R/Y)\ll e^{-c\omega(R)/4}$. This proves the unconditional bound
\begin{equation}
 P_X(R)=\frac{\vol(B_X)}{\zeta(2)}R^2+
 O\!\left(\frac{R(\log\log R)^{1/10}}{(\log R)^{3/10}}\right).
\end{equation}
The classical zero-free region of de la Vallée Poussin
\cite{DelaValleePoussin:1899} gives
$M_\mu(t)\ll t e^{-c\sqrt{\log t}}$; see
\cite[Eq.~(6.17), p.~182]{MontgomeryVaughan:2007}. Taking
$Y=\exp(a\sqrt{\log R})$ for sufficiently small fixed $a>0$ gives
the error $O(R/(\log R)^{1/4})$. 

The prime number theorem alone gives the weaker bound $M_\mu(t)=o(t)$,
which suffices to prove an error term of $o(R)$. Indeed, in
\eqref{eq:primitive_count_cutoff}, first let $R\to\infty$ with $Y$ fixed.
Since $\Delta(R/Y)\to0$, the absolute error divided by $R$ has limsup
$O((\log Y)^{-1/2})$. Letting $Y\to\infty$ proves the claim.
\end{remark}

Finally we note that Theorem \ref{thm:main} with \begin{equation}C=\frac{1}{12}\frac{\vol{(B_X)}}{\zeta(2)}4=\frac{2\vol{(B_X)}}{\pi^2}\end{equation} follows from Theorem~\ref{thm:primitive_count} and Remark~\ref{rem:unconditional_count}, together with \eqref{eq:markov_primitive_count} and \eqref{eq:markov_geodesic_count}.

\section*{Declaration of AI use}
The author used ChatGPT/Codex (OpenAI) and Claude (Anthropic) to assist with exposition, literature search,
exploring and checking mathematical arguments, and developing code
for computations and figures. The author independently verified and
revised all resulting material extensively, and takes full responsibility for
the final manuscript, including its mathematical arguments,
numerical results, and references.

\bibliographystyle{amsplain}
\providecommand{\bysame}{\leavevmode\hbox to3em{\hrulefill}\thinspace}
\providecommand{\MR}{\relax\ifhmode\unskip\space\fi MR }
\providecommand{\MRhref}[2]{%
  \href{http://www.ams.org/mathscinet-getitem?mr=#1}{#2}
}
\providecommand{\href}[2]{#2}

\enlargethispage{2\baselineskip}

\end{document}